\documentclass[11pt,leqno]{amsart}

\usepackage{epsfig}
\usepackage{amssymb}
\usepackage{amscd}
\usepackage[matrix,arrow]{xy}
\usepackage{graphicx}
\usepackage{xcolor}
\usepackage{enumitem}
\usepackage{mathtools}
\usepackage{comment}

\newtheorem*{Thm*}{Theorem}
\newtheorem*{Prop*}{Proposition}
\newtheorem*{Cor*}{Corollary}
\newtheorem{Thm}{Theorem}[section]
\newtheorem{Prop}[Thm]{Proposition}

\newtheorem{Lem}[Thm]{Lemma}
\numberwithin{equation}{section}

\theoremstyle{definition}
\newtheorem{Conj}[Thm]{Conjecture}
\newtheorem{Def}[Thm]{Definition}
\newtheorem{eg}[Thm]{Example}
\newtheorem{Ques}[Thm]{Question}
\newtheorem{Rmk}[Thm]{Remark}

\def\N{\mathbb{N}}
\def\R{\mathbb{R}}
\def\C{\mathbb{C}}
\def\Z{\mathbb{Z}}
\def\Q{\mathbb{Q}}
\def\P{\mathbb{P}}
\def\H{\mathbb{H}}

\def\D{\mathcal{D}}
\def\O{\mathcal{O}}
\def\cC{\mathcal{C}}
\def\cM{\mathcal{M}}
\def\cA{\mathcal{A}}
\def\mc {\mathcal}

\def\Stab{\mathrm{Stab}}
\def\Aut{\mathrm{Aut}}

\def\Hom{\mathrm{Hom}}
\def\Ob{\mathrm{Ob}}
\def\GL{\mathrm{GL}}
\def\SL{\mathrm{SL}}
\def\Coh{\mathrm{Coh}}

\def\cat{\mathrm{cat}}

\def\cl{\mathrm{cl}}
\def\Mod{\mathrm{Mod}}
\def\Teich{\mathrm{Teich}}
\def\Cone{\mathrm{Cone}}
\def\Pic{\mathrm{Pic}}

\def\ra{\rightarrow}

\def\bs{\backslash}

\title[On pseudo-Anosov autoequivalences]{On pseudo-Anosov autoequivalences}

\author{Yu-Wei Fan}
\email{yuweifanx@gmail.com}
\author{Simion Filip}
\email{sfilip@math.uchicago.edu}
\author{Fabian Haiden}
\email{Fabian.Haiden@maths.ox.ac.uk}
\author{Ludmil Katzarkov}
\email{l.katzarkov@miami.edu}
\author{Yijia Liu}
\email{yijial@ksu.edu}

\subjclass[2010]{Primary 18E30; Secondary 16G20, 14F05, 32G15}
\keywords{triangulated categories, autoequivalences, Bridgeland stability conditions, pseudo-Anosov maps}

\begin{document}

\begin{abstract}
Motivated by results of Thurston, we prove that any autoequivalence of a triangulated category induces a filtration by triangulated subcategories, provided the existence of Bridgeland stability conditions.
The filtration is given by the exponential growth rate of masses under iterates of the autoequivalence, and only depends on the choice of a connected component of the stability manifold.
We then propose a new definition of pseudo-Anosov autoequivalences,
and prove that our definition is more general than the one previously proposed by Dimitrov, Haiden, Katzarkov, and Kontsevich.
We construct new examples of pseudo-Anosov autoequivalences on the derived categories of quintic Calabi--Yau threefolds and quiver Calabi--Yau categories.
Finally, we prove that certain pseudo-Anosov autoequivalences on quiver 3-Calabi--Yau categories  act hyperbolically on the space of Bridgeland stability conditions.
\end{abstract}

\maketitle

\section{Introduction}

Let $\Sigma_g$ be a closed orientable surface of genus $g\geq2$.
The mapping class group $\Mod(\Sigma_g)$ is the group of isotopy classes of orientation-preserving diffeomorphisms of $\Sigma_g$.
Thurston proved a far-reaching classification theorem for elements of $\Mod(\Sigma_g)$, showing that each must be either periodic, reducible, or pseudo-Anosov \cite{Th}.
Among these three types, pseudo-Anosov maps are dynamically the most interesting ones.
Moreover, a generic element in $\Mod(\Sigma_g)$ is pseudo-Anosov \cite{Rivin}.
Pseudo-Anosov maps are pervasive throughout low-dimensional topology, geometry, and dynamics.
For instance, a classical theorem of Thurston \cite{Th2} says that the mapping torus constructed from a diffeomorphism $f\colon \Sigma_g\ra\Sigma_g$ has a hyperbolic structure if and only if $f$ is isotopic to a pseudo-Anosov map.

More recently, a striking series of papers by Gaiotto--Moore--Neitzke \cite{GMN},
Bridgeland--Smith \cite{BS}, Haiden--Katzarkov--Kontsevich \cite{HKK}, and Dimitrov--Haiden--Katzarkov--Kontsevich \cite{DHKK},
established connections between Teichm\"uller theory and the theory of Bridgeland stability conditions on triangulated categories.
A series of analogies emerged, such as the one between geodesics (for flat metrics) and stable objects,
with lengths of the former giving the masses of the later.
Further correspondences are summarized in Table~\ref{Table}.
Based on these connections, the categorical analogue of the notion of systoles on Riemann surfaces has been studied in  \cite{Fan2, Hai}.

\begin{table}[h!] \label{Table}
\begin{tabular}{c|c}
Surface & Triangulated category \\ \hline\hline
diffeomorphisms & autoequivalences \\ \hline
closed curve $C$ & object $E$ \\ \hline
intersections $C_1\cap C_2$  & morphisms $\Hom^\bullet(E_1, E_2)$ \\ \hline
flat metrics & stability conditions \\ \hline
geodesics & stable objects \\ \hline
length($C$) & mass($E$) \\ \hline
slope($C$) & phase($E$)
\end{tabular}
\vspace{5pt}
\caption{Analogies between smooth surfaces and triangulated categories.}
\end{table}

Since generic elements in mapping class groups are of pseudo-Anosov type,
it is then natural to ask, under the correspondence  described in Table \ref{Table},
\begin{center}
\vspace{5pt}
What is the categorical analogue of pseudo-Anosov maps?
\vspace{5pt}
\end{center}
In other words,
\begin{center}
\vspace{5pt}
What properties are satisfied by  \emph{generic} autoequivalences on triangulated categories?
\vspace{5pt}
\end{center}

We study the categorifications of several different characterizations of pseudo-Anosov maps in this article.
Recall that a map $f\colon\Sigma_g\ra\Sigma_g$ is called \emph{pseudo-Anosov} if there exists a pair of transverse measured foliations on $\Sigma_g$ and a real number $\lambda>1$ such that
the foliations are preserved by $f$ and their transverse measures are multiplied by $\lambda$ and $1/\lambda$ respectively.
The number $\lambda$  is called the \emph{stretch factor} of the pseudo-Anosov map.
Thurston gives another characterization of pseudo-Anosov maps:

\begin{Thm}[Thurston \cite{Th}, Theorem 5] \label{Thm:Thurston}
For any diffeomorphism $f$ of a surface $\Sigma$, there is a finite set of algebraic integers
$1\leq\lambda_1<\lambda_2<\cdots<\lambda_k$ such that for any homotopically nontrivial simple closed curve $\alpha$,
there is a $\lambda_i$ such that for any Riemannian metric $g$ on $\Sigma$,
$$
\lim_{n\ra\infty}(l_g(f^n\alpha))^{1/n}=\lambda_i.
$$
Here $l_g$ denotes the length of a shortest representative in the homotopy class.
Moreover, $f$ is isotopic to a pseudo-Anosov map if and only if $k=1$ and $\lambda_1>1$.
In this case, $\lambda_1$ is the stretch factor of the pseudo-Anosov map.
\end{Thm}

There is yet another characterization of pseudo-Anosov maps.
Let $(X,d)$ be a metric space and $f\colon X\ra X$ be an isometry.
The \emph{translation length} of $f$ is defined by
\[
\tau(f) \coloneqq \inf_{x\in X} \{d(x, f(x)) \}.
\]
The isometry $f$ is said to be \emph{hyperbolic} if $\tau(f)>0$ and
$\tau(f)=d(x, f(x))$ for some $x\in X$.
Now let $f\in\Mod(\Sigma_g)$ and recall that $f$ acts on Teichm\"uller space $\Teich(\Sigma_g)$ as an isometry for the Teichm\"uller metric.
Moreover, $f$ is pseudo-Anosov if and only if $f$ acts hyperbolically on $\Teich(\Sigma_g)$.
In this case, the translation length and the stretch factor are related by
$\tau(f) = \log\lambda(f)$.

Motivated by the definition of pseudo-Anosov maps and the correspondence in Table \ref{Table},
Dimitrov, Haiden, Katzarkov, and Kontsevich \cite{DHKK}
proposed  a definition of ``pseudo-Anosov autoequivalence" of a triangulated category.
We recall their definition in Section \ref{pA:auto} and will refer to such autoequivalences as ``DHKK pseudo-Anosov autoequivalences'' throughout this article.
Some properties of these autoequivalences are studied in a recent work of Kikuta \cite{Kikuta}.
In particular, it was proved in \cite{Kikuta} that DHKK pseudo-Anosov autoequivalences act hyperbolically on $\Stab^\dagger(\D)/\C$,
and also satisfy the relation $\tau(f)=\log\lambda(f)$ between translation lengths and stretch factors.

However, a drawback of the definition of DHKK pseudo-Anosov autoequivalences, which we will discuss in Section \ref{eg:quiver}, is that it is too restrictive in some cases.
For instance, the set of DHKK pseudo-Anosov autoequivalences does not contain the autoequivalences 
induced by certain pseudo-Anosov maps 
constructed by Thurston \cite{Th} which act trivially on the homology of $\Sigma_g$.

We propose a new definition of \emph{pseudo-Anosov autoequivalences} on triangulated categories.
Our definition is motivated by the characterization given by Theorem \ref{Thm:Thurston}.
We show in Section \ref{sec:fil} that any autoequivalence on a triangulated category $\D$ induces a filtration by triangulated subcategories,
provided the following two assumptions on $\D$: \\
\textbf{Assumption (A)}: There exists a Bridgeland stability condition on $\D$. \\
\textbf{Assumption (B)}: Under Assumption (A), a connected component $\Stab^\dagger(\D)\subset\Stab(\D)$ is fixed once and for all.

\emph{Throughout the article, we consider triangulated categories satisfying these two assumptions.}
The filtration of $\D$ that we associated to an autoequivalence is the categorical analogue of the growth rates of lengths 
$1\leq\lambda_1<\lambda_2<\cdots<\lambda_k$
in Theorem \ref{Thm:Thurston}.
We then make the following definition.

\begin{Def}[see Definition \ref{Def:pAauto}] \label{Def:pAintro}
Let $\D$ be a triangulated category satisfying Assumptions (A) and (B).
An autoequivalence $\Phi\colon\D\ra\D$ is said to be \emph{pseudo-Anosov} if its associated growth filtration has only one step:
$$
0\subset\D_{\lambda}=\D \text{ and }\lambda>1.
$$
\end{Def}

We prove that this new definition is more general than the definition of DHKK pseudo-Anosov autoequivalences.

\begin{Thm}[see Theorem \ref{Thm:more_general}]
Let $\D$ be a triangulated category satisfying Assumptions (A) and (B),
and let $\Phi\in\Aut(\D)$ be an autoequivalence.
If $\Phi$ is DHKK pseudo-Anosov, then it is pseudo-Anosov.
\end{Thm}

Note that there is a distinguished choice of connected component of $\Stab(\D)$ for several examples of triangulated categories $\D$ that admit Bridgeland stability conditions.
In this case, we choose $\Stab^\dagger(\D)$ in Assumption (B) to be the distinguished connected component.
Examples of such categories $\D$ include:
\begin{itemize}
\item The bounded derived category of coherent sheaves on a smooth complex projective variety $X$, where $X$ is a curve, surface, abelian threefold, Fano threefold, or a quintic threefold \cite{Bri,MacriCurve,BriK3,AB13,BMT14,BMS16,MP15,BMSZ17,Li}.
For these triangulated categories, there is a distinguished connected component of $\Stab(\D^b\Coh(X))$ which contains \emph{geometric stability conditions}.
\item $\D(\Gamma_NQ)$, the finite dimensional derived category of the $N$-Calabi--Yau Ginzburg dg algebra $\Gamma_NQ$ associated to an acyclic quiver $Q$, where $N\geq2$ is an integer \cite{Gin,Keller,KellerSurvey}.
For these triangulated categories, there is a distinguished connected component of $\Stab(\D(\Gamma_NQ))$ which contains stability conditions whose heart $P((0,1[)$ coincides with the canonical heart $\mathcal{H}_{\Gamma_NQ}\subset\D(\Gamma_NQ)$ associated to the quiver.
See Section \ref{eg:quiver} for more details in the case of $A_2$-quiver.
\end{itemize}

We construct new examples of \emph{weak} pseudo-Anosov autoequivalences on the derived categories of coherent sheaves on quintic threefolds and on certain quiver Calabi--Yau categories (see Definition~\ref{Def:pAauto} for the definition of \emph{weak pseudo-Anosov autoequivalences}).
Note that it is not known whether there exists any DHKK pseudo-Anosov autoequivalence on a Calabi--Yau category of dimension greater than one.

\begin{Thm}[see Theorem \ref{Thm:pA_quiver} and Proposition \ref{Prop:no_DHKK}]
Let $\D(\Gamma_NA_2)$ be the $N$-Calabi--Yau category associated to the $A_2$-quiver,
where $N\geq3$ is an odd integer.
Then
\begin{itemize}
	\item Any composition of spherical twists $T_1$ and $T_2^{-1}$ that is neither $T_1^a$ nor $T_2^{-b}$ is a weak pseudo-Anosov autoequivalence of $\D(\Gamma_NA_2)$.
	\item There is no DHKK pseudo-Anosov autoequivalence of $\D(\Gamma_NA_2)$.
\end{itemize}
\end{Thm}

We also prove that certain ``palindromic" pseudo-Anosov autoequivalences of $\D(\Gamma_NA_2)$ act hyperbolically on $\Stab^\dagger(\D(\Gamma_NA_2))/\C$ and satisfy the relation between translation lengths and stretch factors
(see Theorem \ref{Thm:hyperbolic} for the precise statement).

\begin{Thm}[see Theorem \ref{Thm:pA_quintic}]
Let $X$ be a quintic Calabi--Yau hypersurface in $\C\P^4$.
Then
\[
\Phi:=T_\O\circ(-\otimes\O(-1))
\]
is a weak pseudo-Anosov autoequivalence on $\D^b(X)$.
Here $T_\O$ denotes the spherical twist with respect to the structure sheaf $\O_X$.
\end{Thm}

In Section~\ref{sec_padiff}, we discuss two ways in which a pseudo-Anosov map in the classical sense induces a pseudo-Anosov autoequivalence on the Fukaya category of a surface.

In the last section, we propose several interesting open questions on categorical dynamical systems that are related to pseudo-Anosov autoequivalences.
In particular, we discuss other possible ways to define the notion of pseudo-Anosov autoequivalences without Assumptions (A) and (B).
We also introduce the notion of \emph{irreducible} autoequivalence of a category in Definition \ref{Def:irred}.
It is then a formal consequence that if an autoequivalence is irreducible and has positive mass growth for a nonzero object, then it is pseudo-Anosov in the sense of Definition \ref{Def:pAintro}.

In Appendix \ref{app:ellcurve}, we show that the mass growth rates with respect to the complexity function (Example~\ref{eg:complexity}) and to stability conditions (Example~\ref{eg:stab}) on the derived category of coherent sheaves on elliptic curves coincide.

\medskip

\noindent{\bf Acknowledgements.}
We thank Hanwool Bae and Sangjin Lee for pointing out several mistakes in the earlier version of the present article.
This research was partially conducted during the period SF served as a Clay Research Fellow.
SF gratefully acknowledges support from the Institute for Advanced Study in Princeton under Grant No. DMS-1638352,
and the Mathematical Sciences Research Institute in Berkeley under Grant No. DMS-1440140 during the Fall 2019 semester.
This material is based upon work supported by the National Science Foundation, as well as grants DMS-1107452, 1107263, 1107367 ``RNMS: Geometric Structures and Representation Varieties'' (the GEAR Network).
LK was supported by a Simons research grant, NSF DMS 150908, ERC Gemis,
DMS-1265230, DMS-1201475 OISE-1242272 PASI. Simons collaborative Grant
- HMS. HSE-grant, HMS
and automorphic forms.
LK  is partially supported by Laboratory of
Mirror Symmetry NRU
HSE, RF Government grant, ag.  14.641.31.0001.
LK was also supported by National Science Fund of Bulgaria, National Scientific Program ``Excellent Research and People for the Development of European Science" (VIHREN), Project No. KP-06-DV-7.
Finally, we would like to thank an anonymous referee for helpful comments and suggestions.

\section{Growth filtrations and pseudo-Anosov autoequivalences}
\label{sec:section2def}

In this section, we develop the notion of growth filtration associated to any autoequivalence of a triangulated category $\D$ which satisfies Assumptions (A) and (B) in the introduction.
We then use it to define the notion of pseudo-Anosov autoequivalences.
Some examples will be worked out in Section~\ref{sec:examples}.

\subsection{Mass functions on triangulated categories} \label{sec:def}

In order to define the mass growth  of an object with respect to an autoequivalence,
one needs to impose a \emph{mass function} on the underlying triangulated category $\D$.

\begin{Def} \label{Def:Massfn}
Let $\D$ be a triangulated category.
A \emph{mass function} on $\D$ is a non-negative function on the objects
$$
\mu\colon\Ob(\D)\ra\R_{\geq0}
$$
such that
	\begin{enumerate}
	\item $\mu(A)+\mu(C)\geq\mu(B)$ for any exact triangle $A\ra B\ra C\ra A[1]$ in $\D$.
	\item $\mu(E)=\mu(E[1])$ for any $E\in\D$.
	\item $\mu(E)\leq\mu(E\oplus F)$ for any $E,F\in\D$.
	\item $\mu(0)=0$.
	\end{enumerate}
\end{Def}

One can find several natural examples of mass functions on triangulated categories.

\begin{eg} \label{eg:Hom}
Suppose $\D$ is $k$-linear and of finite type, i.e.~for any pair of objects $E,F\in\D$ the sum
$\sum_{i\in\Z}\dim_k\Hom_\D(E, F[i])$ is finite.
Then for any nonzero object $E\in\D$,
$$
\mu_E(F)\coloneqq\sum_{i\in\Z}\dim_k\Hom_\D(E, F[i])
$$
defines a mass function on $\D$.
Note that in this example we may have $\mu_E(F)=0$ for a nonzero object $F\in\D$.
\end{eg}

\begin{eg} \label{eg:complexity}
Let $G$ be a split generator of a triangulated category $\D$.
Then the complexity function considered in \cite{DHKK} defined as
$$
\delta_G(E)\coloneqq
\inf\left\{ \displaystyle k \in \N \,\middle|\,
\begin{xy}
(0,5) *{0}="0", (20,5)*{\star}="1", (30,5)*{\dots}, (40,5)*{\star}="k-1", 
(60,5)*{E\oplus\star}="k",
(10,-5)*{G[n_{1}]}="n1", (30,-5)*{\dots}, (50,-5)*{G[n_{k}]}="nk",
\ar "0"; "1"
\ar "1"; "n1"
\ar@{-->} "n1";"0" 
\ar "k-1"; "k" 
\ar "k"; "nk"
\ar@{-->} "nk";"k-1"
\end{xy}
\, \right\}
$$
defines a mass function on $\D$.
\end{eg}

The main examples of mass functions that we will consider are the ones given by Bridgeland stability conditions.
We recall the definition and some basic facts about Bridgeland stability conditions.

\begin{Def}[Bridgeland \cite{Bri}] \label{def:stab}
Fix a triangulated category $\D$ and a group homomorphism $\cl\colon K_0(\D)\ra \Gamma$
to a finite rank abelian group $\Gamma$.
A Bridgeland stability condition $\sigma=(Z,P)$ on $\D$ consists of:
	\begin{itemize}
	\item a group homomorphism $Z\colon\Gamma\ra\C$, and
	\item a collection of full additive subcategories $P=\{P(\phi)\}_{\phi\in\R}$ of $\D$,
	\end{itemize}
such that:
	\begin{enumerate}
	\item $Z(E)\coloneqq Z(\cl([E]))\in\R_{>0}\cdot e^{i\pi\phi}$ for any $0\neq E\in P(\phi)$;
	\item $P(\phi+1)=P(\phi)[1]$;
	\item If $\phi_1>\phi_2$ and $A_i\in P(\phi_i)$ then $\Hom(A_1, A_2)=0$;
	\item For any $0\neq E\in\D$, there exists a (unique) collection of exact triangles
\[
\xymatrix@C=.5em{
& 0 \ar[rrrr] &&&& E_1 \ar[rrrr] \ar[dll] &&&& E_2
\ar[rr] \ar[dll] && \ldots \ar[rr] && E_{n-1}
\ar[rrrr] &&&& E \ar[dll]  &  \\
&&& A_1 \ar@{-->}[ull] &&&& A_2 \ar@{-->}[ull] &&&&&&&& A_n \ar@{-->}[ull] 
}
\]
	such that $A_i\in P(\phi_i)$ and $\phi_1>\phi_2>\cdots>\phi_n$.
	Denote $\phi_\sigma^+(E)\coloneqq\phi_1$ and $\phi_\sigma^-(E)\coloneqq\phi_n$.
	The \emph{mass} of $E$ with respect to $\sigma$ is defined to be
	$$m_\sigma(E)\coloneqq\sum_{i=1}^n|Z(A_i)|.$$

	\item (Support property \cite{KS}) There exists a constant $C>0$ and a norm $||\cdot||$ on $\Gamma\otimes_\Z\R$ such that
	$$\|\cl([E])\| \leq C|Z(E)|$$ for any $E\in P(\phi)$ and any $\phi\in\R$.
	\end{enumerate}
\end{Def}

The group homomorphism $Z\colon\Gamma\ra\C$ is called the \emph{central charge}.
Nonzero objects in $P(\phi)$ are called \emph{$\sigma$-semistable object of phase $\phi$}.
The sequence in $(4)$ is called the Harder--Narasimhan filtration of $E$
and $A_1,\ldots,A_n$ are called the semistable factors of $E$.

Denote by $\Stab(\D)$ (or more precisely $\Stab_\Gamma(\D)$) the set of Bridgeland stability conditions on $\D$ with respect to $\cl\colon K_0(\D)\ra \Gamma$.
There is a useful topology on $\Stab(\D)$ introduced by Bridgeland, which is induced by the
generalized metric \cite[Proposition 8.1]{Bri}:
\[
d(\sigma_1, \sigma_2) \coloneqq \sup_{E\neq0} 
\left\{
|\phi_{\sigma_2}^-(E)-\phi_{\sigma_1}^-(E)|, \,
|\phi_{\sigma_2}^+(E)-\phi_{\sigma_1}^+(E)|, \,
\left|\log\frac{m_{\sigma_2}(E)}{m_{\sigma_1}(E)}\right|
\right\}
\in[0,\infty].
\]
The forgetful map
\[
\Stab(\D) \ra \Hom(\Gamma, \C), \ \ \ \sigma=(Z, P) \mapsto Z
\]
is a local homeomorphism \cite[Theorem 7.1]{Bri}.
Hence $\Stab(\D)$ is naturally a finite dimensional complex manifold.

There are two natural group actions by $\Aut(\D)$ and $\widetilde{\GL^+(2,\R)}$ on the space of Bridgeland stability conditions which commute with each other \cite[Lemma 8.2]{Bri}.
The group of autoequivalences $\Aut(\D)$ acts on $\Stab(\D)$ by isometries with respect to the generalized metric.
To describe the action explicitly, let $\Phi\in\Aut(\D)$ be an autoequivalence and define
\[
\sigma=(Z,P) \mapsto \Phi\cdot\sigma \coloneqq (Z\circ\Phi^{-1}, P'),
\]
where $P'(\phi)=\Phi(P(\phi))$.

Let $\pi\colon\widetilde{\GL^+(2,\R)}\ra\GL^+(2,\R)$ be the universal cover.
Recall that $\widetilde{\GL^+(2,\R)}$ is isomorphic to the group of pairs $(T,f)$
where $f\colon\R\ra\R$ is an increasing map with $f(\phi+1)=f(\phi)+1$,
and $T\in\GL^+(2,\R)$ such that the induced maps on 
$S^1\cong\R/2\Z\cong(\R^2\bs\{0\})/\R_{>0}$ coincide.
Given $g=(T,f)\in\widetilde{\GL^+(2,\R)}$, define its action on $\Stab(\D)$ as
\[
\sigma=(Z,P) \mapsto \sigma\cdot g \coloneqq (T^{-1}\circ Z, P''),
\]
where $P''(\phi)=P(f(\phi))$.
Note that the subgroup $\C\subset\widetilde{\GL^+(2,\R)}$ acts freely and isometrically on $\Stab(\D)$.
For $z\in\C$, its action on $\Stab(\D)$ is given by
\[
\sigma=(Z,P) \mapsto \sigma\cdot z\coloneqq\left(\exp(-i\pi z)Z, P'''\right),
\]
where $P'''(\phi)=P(\phi+\mathrm{Re}z)$.

Each Bridgeland stability condition on $\D$ defines a mass function on $\D$.

\begin{eg} \label{eg:stab}
Let $\sigma$ be a Bridgeland stability condition on $\D$.
By Ikeda \cite[Proposition 3.3]{Ikeda}, we have
\[
m_\sigma(B)\leq m_\sigma(A)+m_\sigma(C)
\]
for any exact triangle $A\ra B\ra C\ra A[1]$ in $\D$.
Hence the mass $m_\sigma$ defines a mass function on $\D$ in the sense of Def.~\ref{Def:Massfn}.
\end{eg}

\begin{Rmk}
The three examples of mass functions above are all analogues of classical objects in Teichm\"uller theory.
Example \ref{eg:Hom} is an analogue of the geometric intersection number between two curves, see \cite[Lemma 2.19]{DHKK}.
Example \ref{eg:complexity} is an analogue of word length of a curve,
and Example \ref{eg:stab} is an analogue of Riemannian length of a curve.
\end{Rmk}

Let $\D'$ be a thick triangulated subcategory of $\D$, i.e.~$\D'$ is closed under shifts, mapping cones, 
and direct summands.
We denote the space of \emph{relative} mass functions by
\[
\cM(\D, \D') \coloneqq \{ \mu \textrm{ mass function on }\D \mid \, \mu(E)=0 \iff E\in\D' \},
\]
and denote $\cM(\D)\coloneqq\cM(\D, 0)$ the space of mass functions that vanish only at the zero object.
This is the analogue of the space of length functions on Riemann surfaces in Teichm\"uller theory.
We define a generalized metric on $\cM(\D, \D')$ in a similar way as \cite[Proposition 8.1]{Bri}.

\begin{Def}
Let $\cM(\D, \D')$ be the space of mass functions on $\D$ that vanish on $\D'$. Then
\[
d(\mu_1, \mu_2) \coloneqq \sup_{E\notin\D'} \Big\{
\Big|\log\frac{\mu_1(E)}{\mu_2(E)}\Big|
\Big\} \in [0,\infty]
\]
defines a generalized metric on $\cM(\D, \D')$.
\end{Def}

We will only use $\cM(\D)$ in the main text of this article.
In Appendix \ref{Append:Mass} we discuss some basic properties of mass functions: functoriality and induced relative mass functions.

Observe that $\Aut(\D)$ and $\R_{>0}$ act isometrically on the space of mass functions $\cM(\D)$,
and the actions by $\R_{>0}$ are free.
We also consider the quotient spaces $\Stab(\D)/\C$ and $\cM(\D)/\R_{>0}$.
Note that the orbits of the $\C$-action (resp.~$\R_{>0}$-action) in $\Stab(\D)$ (resp.~$\cM(\D)$)
are closed. Hence the quotient (generalized) metrics on $\Stab(\D)/\C$ and $\cM(\D)/\R_{>0}$ are given by
\[
\overline{d}(\overline\sigma_1, \overline\sigma_2)\coloneqq\inf_{z\in\C} d(\sigma_1, \sigma_2\cdot z) \ \ \ \mathrm{and} \ \ \
\overline{d}(\overline\mu_1, \overline\mu_2)\coloneqq\inf_{r\in\R_{>0}} d(\mu_1, \mu_2\cdot r).
\]

\begin{Lem}
The map $\Stab(\D)\ra\cM(\D)$ which sends $\sigma$ to its mass $m_\sigma$
descends to a map $\Stab(\D)/\C\ra\cM(\D)/\R_{>0}$.
Moreover, both maps are contractions between (generalized) metric spaces, therefore continuous.
\end{Lem}

\begin{proof}
This follows easily from the definitions of the (generalized) metrics on $\Stab(\D)$  and $\cM(\D)$.
\end{proof}

We denote the image of these two maps by $\cM_\Stab(\D)$ and $\cM_\Stab(\D)/\R_{>0}$ respectively.
Note that the map $\Stab(\D)/\C\ra\cM(\D)/\R_{>0}$ is not injective in general.
For instance, let $\D$ be the derived category of $\C$-linear representations of the $A_2$-quiver.
Let $E_1$ and $E_2$ be the simple representations with dimension vectors $(1,0)$ and $(0,1)$.
Choose any Bridgeland stability condition $\sigma\in\Stab(\D)$ such that
$0<\phi(E_1)<\phi(E_2)<1$ and $|Z(E_1)|=|Z(E_2)|=1$.
Such Bridgeland stability conditions induce the same mass function on $\D$, 
but they are not in the same $\C$-orbit.

\subsection{Growth filtrations of autoequivalences} \label{sec:fil}
In this section, we define the growth filtration associated to an autoequivalence on a triangulated category $\D$
which satisfies Assumptions (A) and (B) in the introduction.

\begin{Def}
Let $\mu$ be a mass function on a triangulated category $\D$,
and $\Phi$ be an autoequivalence on $\D$.
The mass growth of an object $E$ with respect to both $\mu$ and $\Phi$ is defined to be
$$
h_{\mu,\Phi}(E)\coloneqq\limsup_{n\ra\infty}\frac{1}{n}\log\mu(\Phi^nE).
$$
\end{Def}

The following lemma follows directly from the definition of mass functions.

\begin{Lem}
Let $\mu$ and $\Phi$ be as above.
Then
$$
\D_\lambda\coloneqq\{E:h_{\mu,\Phi}(E)\leq\log\lambda\}
$$
is a $\Phi$-invariant thick triangulated subcategory of $\D$.
\end{Lem}

We now restrict our attention to the mass functions given by Bridgeland stability conditions.
The following proposition is a special case of Ikeda \cite[Proposition 3.10]{Ikeda}.
It can also be deduced from a result of Bridgeland \cite[Proposition 8.1]{Bri} stating that if 
a stability condition $\tau\in B_\epsilon(\sigma)$ is in a neighborhood of another stability condition $\sigma$ for small enough $\epsilon>0$, then there are constants $C_1, C_2>0$ such that
\[
C_1 m_\tau(E) < m_\sigma(E) < C_2 m_\tau(E)
\]
for any nonzero object $E\in\D$.

\begin{Prop}
\label{prop:sameconnectedcomponent}
If $\sigma$ and $\tau$ lie in the same connected component of $\Stab(\D)$,
then $h_{m_\sigma, \Phi}(E)=h_{m_\tau, \Phi}(E)$ for any autoequivalence $\Phi$ and any object $E\in\D$.
\end{Prop}

\begin{Def} \label{def:filtration}
Let $\D$ be a triangulated category such that $\Stab(\D)\neq\emptyset$.
We fix a choice of a connected component $\Stab^\dagger(\D)$ of $\Stab(\D)$, and define the \emph{growth filtration} associated to an autoequivalence $\Phi\in\Aut(\D)$ to be the filtration $\{\D_\lambda\}$ by $\Phi$-invariant thick triangulated subcategories of $\D$ with respect to the mass functions given by the stability conditions in $\Stab^\dagger(\D)$.
\end{Def}

The possible mass growth rates $\lambda$ are the categorical analogue of the growth rates of lengths 
$1\leq\lambda_1<\lambda_2<\cdots<\lambda_k$
in Thurston's theorem (Theorem \ref{Thm:Thurston}).
However, we do not know whether there are only finitely many possible mass growth rates for an autoequivalence in general.
The algebraicity of $\lambda$ is unclear in general as well.
The problem of algebraicity of categorical entropy was also raised in \cite[Question 4.1]{DHKK}.

Note that one can also define the growth filtration (therefore define the notion of pseudo-Anosov autoequivalences as in the next section) using other mass functions, for instance the complexity functions in Example \ref{eg:complexity} or the Ext-distance functions in Example \ref{eg:Hom}. See Section \ref{sec:openquestions} and Appendix \ref{app:ellcurve} for more discussions on these alternative definitions. In the main text, we will stick with the growth filtration defined in Definition \ref{def:filtration}.

\subsection{Pseudo-Anosov autoequivalences} \label{pA:auto}

From now on, we consider triangulated categories $\D$ satisfying\\
\textbf{Assumption (A)}: $\Stab(\D)\neq\emptyset$, and \\
\textbf{Assumption (B)}: we fix a connected component $\Stab^\dagger(\D)\subset\Stab(\D)$. 

Then any autoequivalence $\Phi\in\Aut(\D)$ has an associated growth filtration  $\{\D_\lambda\}$ of $\D$ as in Definition \ref{def:filtration}.
Motivated by Thurston's Theorem~\ref{Thm:Thurston},
we propose the following definition of pseudo-Anosov autoequivalences.

\begin{Def} \label{Def:pAauto}
An autoequivalence $\Phi$ is said to be \emph{pseudo-Anosov} if the associated growth filtration has only one step:
$$
0\subset\D_{\lambda}=\D
$$
with  $\lambda>1$.
In other words, the mass growth of any nonzero object is the same positive constant $\log\lambda$.
We call $\lambda$ the \emph{stretch factor} of $\Phi$.

An autoequivalence $\Phi$ is said to be \emph{weak pseudo-Anosov} if the mass growth of any object with \emph{nonzero class} (i.e.~its image under $\cl\colon K_0(\D)\rightarrow\Gamma$ is nonzero) is the same positive constant $\log\lambda$.
\end{Def}

In a previous work of Dimitrov, Haiden, Katzarkov, and Kontsevich \cite{DHKK},
another definition of pseudo-Anosov autoequivalences was proposed.
We call such autoequivalences ``DHKK pseudo-Anosov" in this article.

\begin{Def}[\cite{DHKK}, Definition 4.1 and \cite{Kikuta}, Definition 4.6] \label{Def:DHKK}
An autoequivalence $\Phi\in\Aut(\D)$ is said to be \emph{DHKK pseudo-Anosov} if
there exists a Bridgeland stability condition $\sigma\in\Stab^\dagger(\D)$ and an element 
$g\in\widetilde{\GL^+(2,\R)}$ such that
	\begin{itemize}
	\item $\Phi\cdot\sigma = \sigma \cdot g$.
	\item $\pi(g)=
\begin{pmatrix} 
r^{-1} & 0 \\
0 & r
\end{pmatrix} \ \mathrm{or} \
\begin{pmatrix} 
r & 0 \\
0 & r^{-1}
\end{pmatrix}
	\in \GL^+(2, \R)$ for some $|r|>1$.
	\end{itemize}
Here $\lambda\coloneqq|r|>1$ is called the \emph{DHKK stretch factor} of $\Phi$.
\end{Def}

\begin{Rmk}
We use the slightly modified definition of DHKK pseudo-Anosov by Kikuta \cite[Definition~4.6]{Kikuta}.
The only difference from the definition in \cite{DHKK} is that we require the stability condition $\sigma$ to lie in the distinguished connected component $\Stab^\dagger(\D)$.
\end{Rmk}

The definition of DHKK pseudo-Anosov autoequivalences is motivated by the original definition of pseudo-Anosov maps on Riemann surfaces --
there exists a pair of transverse measured foliations that are preserved by the map and the transverse measures are stretched/contracted by the map.
In Definition \ref{Def:DHKK}, the Bridgeland stability condition $\sigma$ plays the role of a pair of measured foliations,
and the two conditions on $(\Phi, \sigma, g)$ correspond to the preservation and stretching/contracting of the foliations by the map.

\begin{Rmk}
A DHKK pseudo-Anosov autoequivalence will preserve the set of semistable objects associated to the corresponding Bridgeland stability condition.
On the other hand, it should act on the Grothendieck group of the category in a non-trivial stretching/contracting manner.
These two requirements make it difficult to find DHKK pseudo-Anosov maps and indeed very few examples are known, see for instance
\cite{DHKK, Kikuta}.
\end{Rmk}

We show that the definition of pseudo-Anosov in Definition~\ref{Def:pAauto} is more general than DHKK pseudo-Anosov.
In the next section, we provide examples of pseudo-Anosov autoequivalences that are not DHKK pseudo-Anosov.

\begin{Thm} \label{Thm:more_general}
Let $\D$ be a triangulated category and $\Phi\in\Aut(\D)$ an autoequivalence.
If $\Phi$ is DHKK pseudo-Anosov, then it is pseudo-Anosov.
Moreover, the stretch factor and the DHKK stretch factor coincide.
\end{Thm}

\begin{proof}
Let $\Phi$ be a DHKK pseudo-Anosov autoequivalence on $\D$.
By definition, there exists a Bridgeland stability condition $\sigma\in\Stab^\dagger(\D)$ and an element 
$g\in\widetilde{\GL^+(2,\R)}$ satisfying the conditions in Definition \ref{Def:DHKK}.

We claim that the mass growth of any nonzero object $E\in\D$ with respect to $\Phi$ and $\sigma$ is $\log\lambda>0$ where $\lambda>1$ is the DHKK stretch factor.
Let $A_1, A_2, \ldots, A_k$ be the Harder--Narasimhan semistable factors of $E$ with respect to the stability condition $\sigma\in\Stab^\dagger(\D)$.
Since $\Phi\cdot\sigma = \sigma \cdot g$,
the set of $\sigma$-semistable objects and the order of their phases are preserved by $\Phi$.
Hence the Harder--Narasimhan semistable factors of $\Phi^nE$ are given by
$\Phi^nA_1, \Phi^nA_2, \ldots, \Phi^nA_k$, so
\[
m_\sigma(\Phi^nE) = \sum_{i=1}^k |Z_\sigma(\Phi^nA_k)|.
\]
Thus, to prove the claim, it suffices to show that
\[
\lim_{n\ra\infty} \frac{1}{n}\log |Z_\sigma(\Phi^nA)| = \log\lambda
\]
for any $\sigma$-semistable object $A\in\D$.

By the conditions in Definition \ref{Def:DHKK},
the limit $\lim_{n\ra\infty} \frac{1}{n}\log |Z_\sigma(\Phi^nA)|$ is either $\log\lambda$
or $-\log\lambda$.
If there were a $\sigma$-semistable object $A\in\D$ such that 
$\lim_{n\ra\infty} \frac{1}{n}\log |Z_\sigma(\Phi^nA)|=-\log\lambda$,
then the central charges of the semistable objects $A, \Phi A, \Phi^2A, \ldots$ would tend to zero,
which contradicts with the support property of stability conditions (Definition \ref{def:stab}).
This concludes the proof.
\end{proof}

We conclude this section by recalling the definition and some properties of categorical entropy of an autoequivalence introduced by Dimitrov--Haiden--Katzarkov--Kontsevich \cite{DHKK}.
This will be useful for the discussions in the next section.

\begin{Def}[\cite{DHKK}, Definition 2.5]
\label{def:catentropy}
Let $\D$ be a triangulated category with a split generator $G$
and let $\Phi\colon\D\ra\D$  be an endofunctor.
The categorical entropy of $\Phi$ is defined to be
\[
h_\cat(\Phi)\coloneqq\lim_{n\ra\infty} \frac{1}{n} \log \delta_G(\Phi^nG)\in[-\infty,\infty) ,
\]
where $\delta_G$ is the complexity function considered in Example \ref{eg:complexity}.
\end{Def}

The relation between categorical entropy and mass growth with respect to Bridgeland stability conditions has been studied by Ikeda \cite{Ikeda}.

\begin{Thm}[\cite{Ikeda}, Theorem 3.5] \label{Thm:Ikeda}
Let $\Phi\colon\D\ra\D$ be an endofunctor and $\sigma$ be a Bridgeland stability condition on $\D$.
Assume that $\D$ has a split generator $G\in\D$. Then
\[
\sup_{E\in\D} \Big\{
\limsup_{n\ra\infty} \frac{1}{n} \log m_\sigma(\Phi^nE) \Big\} =
\limsup_{n\ra\infty} \frac{1}{n} \log m_\sigma(\Phi^nG) \leq
h_\cat(\Phi).
\]
\end{Thm}

In particular, if $\Phi\in\Aut(\D)$ is pseudo-Anosov, then its categorical entropy is positive since
\[
h_\cat(\Phi) \geq \limsup_{n\ra\infty} \frac{1}{n} \log m_\sigma(\Phi^nG) = \log\lambda >0.
\]

\section{Examples}
\label{sec:examples}
In this section, we construct examples of weak pseudo-Anosov autoequivalences on certain quiver Calabi--Yau categories as well as the derived category  of  coherent sheaves on quintic Calabi--Yau threefolds.

\subsection{Examples from $A_2$-quivers} \label{eg:quiver}
We study the pseudo-Anosovness of autoequivalences on the Calabi--Yau category $\D(\Gamma_NA_2)$ of dimension $N\geq3$ associated to an $A_2$-quiver:
\[
\xymatrix{
\bullet_{_1}\ar[r]&\bullet_{_2} 
}
\]
The category $\D(\Gamma_NA_2)$ is defined to be the derived category of finite dimensional dg-modules over the Ginzburg $N$-Calabi--Yau dg-algebra of the $A_2$-quiver.
We recall some basic properties of $\D(\Gamma_NA_2)$ and refer to \cite{Gin,Keller,KellerSurvey,KQ15} for the definition.
By \cite[Theorem~6.3]{Keller}, $\D(\Gamma_NA_2)$ is an $N$-Calabi--Yau triangulated category, i.e.~for any pair of objects $E,F\in\D(\Gamma_NA_2)$, there are natural isomorphisms
\[
\Hom^\bullet(E,F) \cong \Hom^\bullet(F, E[N])^\vee.
\]
The Ginzburg dg algebra $\Gamma_NA_2$ is concentrated in negative degrees, therefore its derived category has a canonical $t$-structure given by the standard truncation functors. This induces a canonical $t$-structure on $\D(\Gamma_NA_2)$.

\begin{Prop}[{\cite[Section~2]{Amiot}},{\cite[Section~5.1]{KellerYang}}]
\label{prop:amiot}
There is a canonical bounded t-structure on $\D(\Gamma_NA_2)$ with the heart $\mathcal{H}_{\Gamma_NA_2}$ such that the zeroth homology functor $H_0\colon\D(\Gamma_NA_2)\ra\mathrm{mod}\text{-}kA_2$ induces an equivalence of abelian categories:
\[
H_0\colon\mathcal{H}_{\Gamma_NA_2}\xrightarrow{\sim}\mathrm{mod}\text{-}kA_2.
\]
Here $\mathrm{mod}\text{-}kA_2$ denotes the category of finite dimensional $kA_2$-modules, where $kA_2$ is the path algebra of the $A_2$-quiver.
\end{Prop}

Therefore, the heart $\mathcal{H}_{\Gamma_NA_2}$ is generated by the simple $\Gamma_NA_2$-modules $S_1,S_2$ which correspond to the two vertices of the $A_2$-quiver; and $\D(\Gamma_NA_2)$ also is generated by $S_1$, $S_2$ (and their shifts).
By \cite[Lemma~4.4]{Keller}, both $S_1,S_2$ are spherical objects in the $N$-Calabi--Yau category $\D(\Gamma_NA_2)$, namely, they satisfy
\[
\Hom^\bullet(S_1,S_2)=\C[-1] \text{ \ and \ }
\Hom^\bullet(S_i,S_i)=\C\oplus\C[-N]
\text{ for } i=1,2.
\]
For any spherical object $S$, one can associate to it an autoequivalence called the (Seidel--Thomas) \emph{spherical twist} $T_S$ \cite{ST}:
\[
T_S(E) = \Cone(\Hom^\bullet(S, E)\otimes S \ra E).
\]
We denote the spherical twist associated to $S_i$ by $T_i$.

The space of Bridgeland stability conditions on $\D(\Gamma_NA_2)$ has been studied in \cite{BQS, BS, Ikeda2}. There is a distinguished connected component of $\Stab(\D(\Gamma_NA_2))$ which contains all stability conditions whose heart $P((0,1])$ coincides with the canonical heart $\mathcal{H}_{\Gamma_NA_2}$ (see for instance \cite[Section 5]{Ikeda2} and references therein).
We choose $\Stab^\dagger(\D(\Gamma_NA_2))$ in Assumption (B) to be the distinguished connected component.

We prove a large class of autoequivalences on $\D(\Gamma_NA_2)$ are weak pseudo-Anosov.

\begin{Thm} \label{Thm:pA_quiver}
Let $N\geq3$ be an odd integer.
Let  $\Phi\in\Aut(\D(\Gamma_NA_2))$ be any composition of $T_1$ and $T_2^{-1}$
(e.g.~$T_1^{a_1}T_2^{-b_1}T_1^{a_2}T_2^{-b_2}\cdots$ for $a_i,b_i\geq0$)
that is neither $T_1^a$ nor $T_2^{-b}$.
Then $\Phi$ is weak pseudo-Anosov.
\end{Thm}

For the ease of notations, we make the following definition.

\begin{Def}
Let $\D$ be a triangulated category and $E,A_1,\ldots,A_n$ be objects in $\D$.
We write
\[
E\in\{A_1,A_2,\ldots,A_n\}
\]
if there exists a sequence of exact triangles
\[
\xymatrix@C=.5em{
& 0 \ar[rrrr] &&&& \star \ar[rrrr] \ar[dll] &&&& \star
\ar[rr] \ar[dll] && \ldots \ar[rr] && \star
\ar[rrrr] &&&& E. \ar[dll]  &  \\
&&& A_1 \ar@{-->}[ull] &&&& A_2 \ar@{-->}[ull] &&&&&&&& A_n \ar@{-->}[ull]
}
\]
\end{Def}

\begin{Lem}
\label{lem:sphericaltwist}
Let $N\geq3$ be an integer. In $\D(\Gamma_NA_2)$, we have:
    \begin{itemize}
        \item $T_1(S_1)=S_1[1-N]$ and $T_1(S_2)\in\{S_2,S_1\}$.
        \item $T_2(S_1)\in\{S_1,S_2[2-N]\}$ and $T_2(S_2)=S_2[1-N]$.
        \item $T_1^{-1}(S_1)=S_1[N-1]$ and $T_1^{-1}(S_2)\in\{S_1[N-2],S_2\}$.
        \item $T_2^{-1}(S_1)\in\{S_2,S_1\}$ and $T_2^{-1}(S_2)=S_2[N-1]$.
    \end{itemize}
\end{Lem}

\begin{proof}
The lemma follows directly from the definition of spherical objects, spherical twists, and the fact that $\Hom^\bullet(S_1,S_2)=\C[-1]$.
For instance, $T_2^{-1}(S_1)\in\{S_2,S_1\}$ can be obtained by applying $T_2^{-1}$ to the exact triangle
$\Hom^\bullet(S_2,S_1)\otimes S_2\ra S_1\ra T_2(S_1) \ra$.
\end{proof}

The following lemma will be useful throughout this section.

\begin{Lem}
\label{lem:swapHNfactors}
Let $N\geq3$ be an integer.
Let $\sigma\in\Stab^\dagger(\D(\Gamma_NA_2))$ be a stability condition such that
$S_1,S_2$ and their shifts are the only indecomposable semistable objects of $\sigma$,
and their phases (with respect to $\sigma$) satisfy
\[
0<\phi(S_1)<\phi(S_2)<1.
\]
Suppose that $E\in\{S_{i_1}^{\oplus q_1}[j_1],S_{i_2}^{\oplus q_2}[j_2]\}$ for some $i_1,i_2\in\{1,2\}$, $q_1,q_2\in\Z_{>0}$, and $j_1,j_2\in\Z$.
Moreover, suppose that 
\[
\phi(S_{i_1})+j_1=\phi(S_{i_1}^{\oplus q_1}[j_1]) < \phi(S_{i_2}^{\oplus q_2}[j_2]) = \phi(S_{i_2})+j_2.
\]
Then
    \begin{enumerate}[label=(\alph*)]
        \item If $i_1\neq i_2$ or $j_2-j_1\neq1$, then $E\in\{S_{i_2}^{\oplus q_2}[j_2],S_{i_1}^{\oplus q_1}[j_1]\}$.
        \item If $i_1=i_2$ and $j_2-j_1=1$, then $E\in\{S_{i_2}^{\oplus q_2-t}[j_2],S_{i_1}^{\oplus q_1-t}[j_1]\}$ for some $t\geq0$.
    \end{enumerate}
In other words, one can exchange the order of $S_{i_1}^{\oplus q_1}[j_1]$ and $S_{i_2}^{\oplus q_2}[j_2]$ so that the phases are in decreasing order, with the caveat that some of the factors might disappear when $i_1=i_2$ and $j_2-j_1=1$.
\end{Lem}

\begin{proof}
If $\phi(S_{i_1}^{\oplus q_1}[j_1]) < \phi(S_{i_2}^{\oplus q_2}[j_2])$, then $\{S_{i_1}^{\oplus q_1}[j_1],S_{i_2}^{\oplus q_2}[j_2]\}$ is of one of the following three types:
\begin{itemize}
\item $\{S_{i}^{\oplus q_1}[j_1],S_{i}^{\oplus q_2}[j_1+n]\}$ for some $i\in\{1,2\}$ and $n>0$,
\item $\{S_{2}^{\oplus q_1}[j_1],S_{1}^{\oplus q_2}[j_1+n]\}$ for some $n>0$,
\item $\{S_{1}^{\oplus q_1}[j_1],S_{2}^{\oplus q_2}[j_1+n]\}$ for some $n\geq0$.
\end{itemize}
Observe that $\Hom(S_{i_2}^{\oplus q_2}[j_2],S_{i_1}^{\oplus q_1}[j_1+1])=0$ holds for all three types, except for $i_1=i_2$ and $j_2-j_1=1$.
Hence if $i_1\neq i_2$ or $j_2-j_1\neq1$, then
$E\cong S_{i_1}^{\oplus q_1}[j_1] \oplus S_{i_2}^{\oplus q_2}[j_2]$.
This proves part (a).

If $i\coloneqq i_1=i_2$ and $j\coloneqq j_1=j_2-1$, then there is an exact triangle
\[
S_{i}^{\oplus q_1}[j] \ra E \ra S_{i}^{\oplus q_2}[j+1] \ra.
\]
Recall that any object $E$ in $\D$ has a unique filtration with respect to the heart of a bounded t-structure $\cA\coloneqq\mathcal{H}_{\Gamma_NA_2}=\left<S_1,S_2\right>_\mathrm{ext}$ on $\D$:
\[
\xymatrix@C=0em{
& 0 \ar[rrrr] &&&& \star \ar[rrrr] \ar[dll] &&&& \star
\ar[rr] \ar[dll] && \ldots \ar[rr] && \star
\ar[rrrr] &&&& E, \ar[dll]  &  \\
&&& H^{-k_1}_\cA(E)[k_1] \ar@{-->}[ull] &&&& H^{-k_2}_\cA(E)[k_2] \ar@{-->}[ull] &&&&&&&& H^{-k_m}_\cA(E)[k_m] \ar@{-->}[ull]
}
\]
where $k_1>k_2>\cdots>k_m$ and $H^{-k_i}_\cA(E)\in\cA$ for each $i$.
The objects $H^{-k_i}_\cA(E)$ are sometimes called the cohomological objects of $E$ with respect to $\cA$. Any exact triangle induces a long exact sequence of cohomological objects.
Hence the exact triangle $S_{i}^{\oplus q_1}[j] \ra E \ra S_{i}^{\oplus q_2}[j+1] \ra$ induces an exact sequence in $\cA$:
\[
0\ra H^{-j-1}_\cA(E) \ra S_{i}^{\oplus q_2} \ra S_{i}^{\oplus q_1} \ra H^{-j}_\cA(E) \ra 0.
\]
Therefore we have
\[
H^{-j-1}_\cA(E) = S_i^{\oplus q_2-t} \text{ \ and \ }
H^{-j}_\cA(E) = S_i^{\oplus q_1-t}
\]
for some $t\geq0$, since $S_i$ is a simple object in $\cA$.
Thus there is an exact triangle
\[
S_i^{\oplus q_2-t}[j+1] \ra E \ra S_i^{\oplus q_1-t}[j] \ra.
\]
This proves part (b).
\end{proof}

We are now ready to prove Theorem \ref{Thm:pA_quiver}.

\begin{proof}[Proof of Theorem \ref{Thm:pA_quiver}]
Since $\mathcal{H}_{\Gamma_NA_2}\cong\mathrm{mod}\text{-}kA_2$ is a finite length abelian category with finitely many simple objects $\{S_1,S_2\}$ (see Proposition~\ref{prop:amiot}), any pair of points $(z_1,z_2)\in\mathbb{H}^2$ in the upper half-plane gives a Bridgeland stability condition on $\D(\Gamma_NA_2)$, such that the central charges of $S_1$ and $S_2$ are $z_1$ and $z_2$, respectively (\cite[Lemma~5.2]{Bri09}).
We choose $z_1,z_2\in\mathbb H$ such that $z_1=e^{i\pi\phi_1}$ and $z_2=e^{i\pi\phi_2}$, where $0<\phi_1<\phi_2<1$, and let $\sigma=(Z,P)\in\Stab^\dagger(\D(\Gamma_NA_2))$ denotes the corresponding stability condition.
Then we have
\[
0<\phi(S_1)=\phi_1<\phi(S_2)=\phi_2<1 \ \ \mathrm{and} \ \ |Z(S_1)|=|Z(S_2)|=1.
\]
The only indecomposable objects in $\mathcal{H}_{\Gamma_NA_2}\cong\mathrm{mod}\text{-}kA_2$ are the simple objects $S_1,S_2$, and an object $E$ given by the extension between these two objects:
\[
0\ra S_2\ra E\ra S_1\ra0.
\]
Since we choose $z_1,z_2$ so that $0<\phi(S_1)<\phi(S_2)<1$, the object $E$ is not semistable with respect to $\sigma$. Therefore, the only indecomposable $\sigma$-semistable objects are $S_1,S_2$, and their shifts.

Let us compute the mass growth of the split generator $S_1\oplus S_2$ of $\D(\Gamma_NA_2)$ with respect to the autoequivalence $\Phi$ and the stability condition $\sigma$. 
By Lemma \ref{lem:sphericaltwist}, since $\Phi$ is a composition of positive powers of $T_1$ and $T_2^{-1}$, for any $n\in\N$ we have
\[
\Phi^n (S_1\oplus S_2)\in \{ S_{i_1}^{\oplus q_1}[(N-1)j_1], S_{i_2}^{\oplus q_2}[(N-1)j_2], \ldots, S_{i_k}^{\oplus q_k}[(N-1)j_k] \}
\]
for some $i_1,\ldots,i_k\in\{1,2\}$, $j_1,\ldots,j_k\in\Z$, and $q_1,\cdots, q_k\in\N$.
By Lemma \ref{lem:swapHNfactors}, one can reorder the sequence so that the objects are of decreasing phases (since $N\geq3$, so case (b) in Lemma \ref{lem:swapHNfactors} does not happen in the sequence).
This gives the Harder--Narasimhan filtration of $\Phi^n (S_1\oplus S_2)$.
By the choice of the stability condition $\sigma$,
the mass of $\Phi^n(S_1\oplus S_2)$ is simply the sum of the powers
$q_1+\cdots+q_k$ in the sequence.
Hence in order to compute the mass growth of $\Phi^n(S_1\oplus S_2)$,
one only needs to compute the increment of the number of $S_1$'s and $S_2$'s 
when $\Phi$ acts on $S_1$ and $S_2$.

The spherical twist $T_1$ corresponds to the matrix
\[
\begin{pmatrix} 
1 & 1 \\
0 & 1 
\end{pmatrix},
\]
i.e.~if there is a sequence with $x_1$ $S_1$'s and $x_2$ $S_2$'s ($S_i^{\oplus q}$ counts as $q$ $S_i$'s),
then after applying $T_1$, the new sequence will have $(x_1+x_2)$ $S_1$'s and $x_2$ $S_2$'s.
This follows from the fact that $T_1(S_1) = S_1[1-N]$ and $T_1(S_2)\in\{S_2,S_1\}$.
Similarly, $T_2^{-1}$  corresponds to the matrix
\[
\begin{pmatrix} 
1 & 0 \\
1 & 1 
\end{pmatrix}.
\]

Let $\Phi=T_1^{a_1}T_2^{-b_1}T_1^{a_2}T_2^{-b_2}\cdots T_1^{a_p}T_2^{-b_p}$ for some $a_i,b_i\geq0$.
Then the corresponding matrix of $\Phi$ is
\[
M_{\Phi} := 
\begin{pmatrix} 
1 & 1 \\
0 & 1 
\end{pmatrix}^{a_1}
\begin{pmatrix} 
1 & 0 \\
1 & 1 
\end{pmatrix}^{b_1}
\cdots
\begin{pmatrix} 
1 & 1 \\
0 & 1 
\end{pmatrix}^{a_p}
\begin{pmatrix} 
1 & 0 \\
1 & 1 
\end{pmatrix}^{b_p} \in \SL(2,\Z).
\]
Since $\Phi$ is neither $T_1^{a}$ nor $T_2^{-b}$,
the corresponding matrix $M_\Phi$ has trace $t\geq3$
and eigenvalues
$
\frac{t\pm\sqrt{t^2-4}}{2}\notin\Q.
$
Since $(1,1)^T$ is not an eigenvector of the eigenvalue $\frac{t-\sqrt{t^2-4}}{2}$,
one can conclude that the mass growth of $S_1\oplus S_2$ is
\[
\lim_{n\ra\infty}\frac{1}{n}\log m_\sigma(\Phi^n(S_1\oplus S_2)) = \log \frac{t+\sqrt{t^2-4}}{2} =: \log\lambda>0.
\]

To prove that $\Phi$ is weak pseudo-Anosov, one needs to show that the mass growth of any object $E$ is also $\log\lambda$ provided $[E]\neq0$ in $K_0(\D)$ .
By Theorem \ref{Thm:Ikeda}, we have
\[
\limsup_{n\ra\infty} \frac{1}{n} \log |Z_\sigma(\Phi^nE)| \leq
\limsup_{n\ra\infty} \frac{1}{n} \log m_\sigma(\Phi^nE) \leq
\lim_{n\ra\infty}\frac{1}{n}\log m_\sigma(\Phi^n(S_1\oplus S_2)) = \log\lambda.
\]
On the other hand, since $N$ is an odd number, the actions of $T_1$ and $T_2^{-1}$ on the Grothendieck group $K_0(\D(\Gamma_NA_2))$ is given by the same matrices
$
\begin{pmatrix} 
1 & 1 \\
0 & 1 
\end{pmatrix}
$
and
$
\begin{pmatrix} 
1 & 0 \\
1 & 1 
\end{pmatrix},
$
with respect to the basis $\{[S_1],[S_2]\}\subset K_0(\D(\Gamma_NA_2))$ (cf.~Lemma~\ref{lem:sphericaltwist}).-
By \cite{Bri}, one can deform $\sigma$ if necessary so that the kernel of $Z_\sigma$ does not contain any eigenvector of $M_\Phi$.
Hence the (exponential) growth of $|Z_\sigma(\Phi^nE)|$ is $\log\lambda$ unless $[E]\in K_0(\D(\Gamma_NA_2))$ is an eigenvector of the smaller eigenvalue 
\[
\lambda' = \frac{t-\sqrt{t^2-4}}{2},
\]
which is impossible since $[E]$ is a nonzero integral vector.
Hence
\[
\limsup_{n\ra\infty} \frac{1}{n} \log |Z_\sigma(\Phi^nE)| =
\limsup_{n\ra\infty} \frac{1}{n} \log m_\sigma(\Phi^nE) =
\log\lambda
\]
for any  object $E$ with $[E]\neq0$,
therefore $\Phi$ is a weak pseudo-Anosov autoequivalence.
\end{proof}

\begin{Rmk}
The same argument also works for compositions of $T_1, T_2^{-1},\ldots,
T_{2n-1}, T_{2n}^{-1}$ on the $N$-Calabi--Yau category associated to the
$A_{2n}$-quiver ($N\geq3$ odd):
\[
\xymatrix{
\bullet_{_1}\ar[r]&\bullet_{_2} & \ar[l]\bullet_{_3}\ar[r]&\cdots
&\ar[l]\bullet_{_{2n-1}}
\ar[r]&\bullet_{_{2n}},
}
\]
provided there is no integral vector in the span of the eigenvectors of eigenvalues other than the spectral radius.
The reason is that the action of $T_1, T_2^{-1},\ldots,
T_{2n-1}, T_{2n}^{-1}$
on the Grothendieck group $K_0(\D(\Gamma_NA_{2n}))$
is represented by non-negative matrices.
For instance, one can show that $T_1T_2^{-1}T_3T_4^{-1}$ on
$\D_{4}^N$ is weak pseudo-Anosov.
This autoequivalence has a geometric interpretation as the total monodromy acting on the Fukaya category of the fiber of a Lefschetz fibration whose vanishing cycles are the $2n$ spherical objects $S_1,\ldots,S_{2n}$. 
\end{Rmk}

On the other hand, the pseudo-Anosov autoequivalences we find in Theorem \ref{Thm:pA_quiver} are not
DHKK pseudo-Anosov.
In fact, we prove the following stronger statement.

\begin{Prop} \label{Prop:no_DHKK}
Let $\D$ be a triangulated category.
If for any $\sigma\in\Stab^\dagger(\D)$ there are only finitely many $\sigma$-stable objects (up to shifts),
then $\D$ does not possess any DHKK pseudo-Anosov autoequivalence.
\end{Prop}

\begin{proof}
Let $\Phi\in\Aut(\D)$ be a DHKK pseudo-Anosov autoequivalence
and assume the triple $(\Phi,\sigma,g)$ satisfies the conditions in Definition \ref{Def:DHKK}.
Then $\Phi$ preserves the set of all $\sigma$-stable objects.
Since there are only finitely many $\sigma$-stable objects up to shifts,
there exists some $N>0$ such that $\Phi^N$ preserves all $\sigma$-stable objects up to shifts by even integers.
Note that the classes of $\sigma$-stable objects generate $K_0(\D)$.
Therefore $\Phi^N$ acts trivially on $K_0(\D)$.
This contradicts with the stretching/contracting condition of DHKK pseudo-Anosov autoequivalences.
\end{proof}

Recall that an isometry $f$ of a metric space $(X,d)$ is called \emph{hyperbolic} if
its \emph{translation length}
$
\tau(f)\coloneqq\inf_{x\in X}\{ d(x , f(x)) \}
$
is positive and the infimum is achieved by some $x\in X$.
It is proved by Kikuta \cite{Kikuta} that
any DHKK pseudo-Anosov autoequivalence $\Phi$ acts hyperbolically on $(\Stab^\dagger(\D)/\C, \bar d)$, and $\tau(\Phi)=\log\lambda$.
We note that it is more interesting to study the hyperbolicity on the quotient space $\Stab^\dagger(\D)/\C$ rather than
$\Stab^\dagger(\D)$,
because trivial examples like the shift functors $[k]$ on any triangulated category act hyperbolically on $\Stab^\dagger(\D)$.

We prove the hyperbolicity of the ``palindromic" (see the statement of Theorem \ref{Thm:hyperbolic} for the definition) pseudo-Anosov autoequivalences constructed in Theorem \ref{Thm:pA_quiver}.
Note that the proof is less straightforward than it is for the DHKK pseudo-Anosov autoequivalences in \cite{Kikuta}.
One of the key ingredients in the proof, which is also used in the classical theory of pseudo-Anosov maps,
is the Perron--Frobenius theorem.
It states that a real square matrix with positive entries has a unique largest real eigenvalue $\lambda$ and that its corresponding eigenvector can be chosen to have positive components.
Moreover, the absolute value of any other eigenvalue is strictly less than $\lambda$.

We say an autoequivalence  $\Phi\in\Aut(\D(\Gamma_NA_2))$ is a \emph{palindromic composition} of $T_1$ and $T_2^{-1}$ if it can be written as
$$
\Phi=T_1^{a_1}T_2^{-b_1}T_1^{a_2}T_2^{-b_2}\cdots T_1^{a_p}, \text{  \quad such that } 
(a_1,b_1,\ldots,a_p)= (a_p,b_{p-1},a_{p-1},\ldots,b_1,a_1).
$$
As before, we consider the following matrix
\[
M_{\Phi} \coloneqq 
\begin{pmatrix} 
1 & 1 \\
0 & 1 
\end{pmatrix}^{a_1}
\begin{pmatrix} 
1 & 0 \\
1 & 1 
\end{pmatrix}^{b_1}
\cdots
\begin{pmatrix} 
1 & 1 \\
0 & 1 
\end{pmatrix}^{a_p}.
\]

\begin{Thm} \label{Thm:hyperbolic}
Let $N\geq3$ be an odd integer.
Let  $\Phi\in\Aut(\D(\Gamma_NA_2))$ be a palindromic composition of $T_1$ and $T_2^{-1}$, which is neither $T_1^a$ nor $T_2^{-b}$.
Then $\Phi$ acts hyperbolically on $\cM_{\Stab^\dagger}(\D(\Gamma_NA_2))/\R^+$, with translation length $\tau(\Phi)=\log\lambda$ where $\lambda$ is the stretch factor of $\Phi$.

If $\Phi$ satisfies an additional condition that  $\log\rho(M_\Phi)\geq\frac{N-1}{2}(\sum a_i+\sum b_i)$, then it also acts hyperbolically on $\Stab^\dagger(\D(\Gamma_NA_2))/\C$, with the same translation length.
\end{Thm}

We prove the following lemma before proving Theorem \ref{Thm:hyperbolic}.

\begin{Lem}
\label{lemma:boundofphase}
Let $\sigma\in\Stab^\dagger(\D(\Gamma_NA_2))$ be a stability condition such that
$S_1,S_2$ and their shifts are the only indecomposable semistable objects of $\sigma$,
and their phases satisfy
$$
0<\phi(S_1)<\phi(S_2)<1.
$$
Suppose that 
\[
E \in \{S_{i_1}^{\oplus q_1}[j_1], \ S_{i_2}^{\oplus q_2}[j_2], \ldots, \  S_{i_k}^{\oplus q_k}[j_k]\},
\]
where $i_1,\ldots,i_k\in\{1,2\}$, $j_1,\ldots,j_k\in\Z$, and $q_1,\ldots,q_k\in\N$.
Then
\[
\min_{1\leq\ell\leq k} \{j_\ell+\phi(S_{i_\ell})\} \leq
\phi_\sigma^-(E) \leq \phi_\sigma^+(E)
\leq \max_{1\leq\ell\leq k} \{j_\ell+\phi(S_{i_\ell})\}
\]
and
\[
m_\sigma(E) \leq \sum_{\ell=1}^k q_\ell |Z_\sigma(S_{i_\ell})|.
\]
\end{Lem}

\begin{proof}
Firstly, if $j_1+\phi(S_{i_1})>\ldots>j_k+\phi(S_{i_k})$, then $\{S_{i_1}^{\oplus q_1}[j_1], \ S_{i_2}^{\oplus q_2}[j_2], \ldots, \  S_{i_k}^{\oplus q_k}[j_k]\}$ gives the Harder--Narasimhan filtration of $E$, hence we have
\[
\phi_\sigma^-(E) =j_k+\phi(S_{i_k}) = \min_{1\leq\ell\leq k} \{j_\ell+\phi(S_{i_\ell})\}, \ 
\phi_\sigma^+(E) =j_1+\phi(S_{i_1})= \max_{1\leq\ell\leq k} \{j_\ell+\phi(S_{i_\ell})\},
\]
and
\[
m_\sigma(E) = \sum_{\ell=1}^k q_\ell |Z_\sigma(S_{i_\ell})|.
\]

Now suppose $j_\ell+\phi(S_{i_\ell})< j_{\ell+1}+\phi(S_{i_{\ell+1}})$ for some $\ell$.
By Lemma \ref{lem:swapHNfactors}, one can exchange the order of $\{S_{i_\ell}^{\oplus q_\ell}[j_\ell], S_{i_{\ell+1}}^{\oplus q_{\ell+1}}[j_{\ell+1}]\}$
unless $i\coloneqq i_\ell=i_{\ell+1}$ and $j\coloneqq j_{\ell}=j_{\ell+1}-1$:
\[
\xymatrix@C=.5em{
& A \ar[rrrr] &&&& B \ar[rrrr] \ar[dll] &&&& C.
 \ar[dll]   \\
&&& S_{i}^{\oplus q_\ell}[j] \ar@{-->}[ull] &&&& S_{i}^{\oplus q_{\ell+1}}[j+1] \ar@{-->}[ull] 
}
\]
In this case, by part (b) of Lemma \ref{lem:swapHNfactors} and the Octahedral Axiom of triangulated categories, there are exact triangles
\[
\xymatrix@C=0em{
& A \ar[rrrr] &&&& \star \ar[rrrr] \ar[dll] &&&& C.
 \ar[dll]   \\
&&&  S_{i}^{\oplus (q_{\ell+1}-t)}[j_\ell+1] \ar@{-->}[ull] &&&&S_{i}^{\oplus (q_\ell-t)}[j_\ell] \ar@{-->}[ull] 
}
\]
for some $t\geq0$.
In other words, one can still exchange the order of these two terms to make the phases decreasing,
but the number of factors of $S_i[j_\ell]$ and $S_i[j_\ell+1]$ might decrease.

In conclusion, one can exchange the order of the factors in $\{S_{i_1}^{\oplus q_1}[j_1], \ S_{i_2}^{\oplus q_2}[j_2], \ldots, \  S_{i_k}^{\oplus q_k}[j_k]\}$ so that the phases are in the decreasing order to get the Harder--Narasimhan filtration of $E$, with the caveat that some of the factors might disappear during the process when exchanging two terms of the type $\{S_{i}^{\oplus q_\ell}[j],  S_{i}^{\oplus q_{\ell+1}}[j+1]\}$.
Hence
\[
\min_{1\leq\ell\leq k} \{j_\ell+\phi(S_{i_\ell})\} \leq
\phi_\sigma^-(E) \leq \phi_\sigma^+(E)
\leq \max_{1\leq\ell\leq k} \{j_\ell+\phi(S_{i_\ell})\}
\]
and
\[
m_\sigma(E) \leq \sum_{\ell=1}^k q_\ell |Z_\sigma(S_{i_\ell})|
\]
as desired.
\end{proof}

Now we prove Theorem \ref{Thm:hyperbolic}.

\begin{proof}[Proof of Theorem \ref{Thm:hyperbolic}]
Recall the \emph{stable translation length} of an isometry $f$ on a metric space $(X,d)$ is defined as
\[
\overline{\tau}(f) \coloneqq \lim_{n\ra\infty} \frac{d(x, f^n(x))}{n}.
\]
One can check by the triangle inequality
that $\overline\tau(f)$ is independent of the choice of $x\in X$,
and $\tau(f)\geq\overline\tau(f)$.
The idea of proof is to find a Bridgeland stability condition $\sigma\in\Stab^\dagger(\D(\Gamma_NA_2))$ such that
\begin{enumerate}[label=(\alph*)]
\item $\log\lambda \geq \bar{d}(\bar\sigma, \Phi\bar\sigma)$, and
\item $\lim_{n\ra\infty}\frac{1}{n}\bar{d}(\bar\sigma, \Phi^n\bar\sigma) \geq \log\lambda$.
\end{enumerate}
Here $\bar d$ is the metric on $\Stab^\dagger(\D(\Gamma_NA_2))/\C$ or $\cM_{\Stab^\dagger}(\D(\Gamma_NA_2))/\R^+$ defined in Section \ref{sec:def}.
Assuming that one can find such a stability condition $\sigma$, then we have
\[
\log\lambda \geq \bar{d}(\bar\sigma, \Phi\bar\sigma) \geq
\tau(\Phi) \geq \overline\tau(\Phi) =
\lim_{n\ra\infty}\frac{1}{n}\bar{d}(\bar\sigma, \Phi^n\bar\sigma) \geq \log\lambda.
\]
Therefore $\tau(\Phi)=\overline\tau(\Phi)=\log\lambda>0$ and the infimum $\tau(\Phi)$ is achieved by 
$\bar{d}(\bar\sigma, \Phi\bar\sigma)$.
Hence the action is hyperbolic.

We choose  the stability condition $\sigma$ as follows.
Let $M_\Phi$ be the corresponding matrix of $\Phi$ we defined in the proof of Theorem \ref{Thm:pA_quiver}:
\[
M_{\Phi} \coloneqq 
\begin{pmatrix} 
1 & 1 \\
0 & 1 
\end{pmatrix}^{a_1}
\begin{pmatrix} 
1 & 0 \\
1 & 1 
\end{pmatrix}^{b_1}
\cdots
\begin{pmatrix} 
1 & 1 \\
0 & 1 
\end{pmatrix}^{a_p}.
\]
Since $\Phi$ is neither $T_1^a$ nor $T_2^{-b}$,
the matrix $M_\Phi$ is a \emph{positive} matrix.
By Perron--Frobenius theorem, the eigenvector $v=[v_1, v_2]^T$ of the larger eigenvalue of $M^T_\Phi$
can be chosen to have positive components $v_1,v_2>0$.
We choose a Bridgeland stability condition $\sigma\in\Stab^\dagger(\D(\Gamma_NA_2))$ such that
$S_1,S_2$ and their shifts are the only indecomposable semistable objects of $\sigma$,
and their phases and central charges satisfy
$$
0<\phi(S_1)<\phi(S_2)<1, \ \ |Z(S_1)|=v_1 \ \ \mathrm{and} \ \ |Z(S_2)|=v_2.
$$
We claim that $\sigma$ chosen in this way satisfies conditions (a) and (b) above
for the metric space $\cM_{\Stab^\dagger}(\D(\Gamma_NA_2))/\R^+$,
and for $\Stab^\dagger(\D(\Gamma_NA_2))/\C$ when $N=3$.

\medskip

\noindent\emph{Proof of (a):}
We first consider the metric space $(\Stab^\dagger(\D(\Gamma_NA_2))/\C, \bar d)$.
Recall that
\begin{align*}
\bar d(\bar\sigma, \Phi\bar\sigma) & = \inf_{z\in\C} d(\sigma, \Phi\sigma\cdot z) \\
 & = \inf_{\substack{x\in\R \\ y>0}}
 \sup_{E\neq0}
\Big\{
|\phi^-_\sigma(E) - \phi^-_\sigma(\Phi E) + x|,
|\phi^+_\sigma(E) - \phi^+_\sigma(\Phi E) + x|,
|\log\frac{ym_\sigma(\Phi E)}{m_\sigma(E)}|
\Big\}.
\end{align*}

Let $E\in\D(\Gamma_NA_2)$ be any nonzero object.
We claim the following inequalities hold:
\begin{enumerate}
\item $\phi_\sigma^\pm(T_1(E))\leq\phi^\pm_\sigma(E)\leq\phi_\sigma^\pm(T_1(E))+(N-1)$.
\item $\phi^\pm_\sigma(T_2^{-1}(E))-(N-1)\leq\phi^\pm_\sigma(E)\leq\phi_\sigma^\pm(T_2^{-1}(E))$.
\end{enumerate}
We will only prove (1) as the proof of (2) is similar.
Suppose the Harder--Narasimhan filtration of $E$ is:
\[
S_{i_1}^{\oplus q_1}[j_1], \ S_{i_2}^{\oplus q_2}[j_2], \ldots, \  S_{i_k}^{\oplus q_k}[j_k],
\]
where $i_1,\ldots,i_k\in\{1,2\}$, $j_1,\ldots,j_k\in\Z$, $q_1,\ldots,q_k\in\N$, and
\[
\phi_\sigma^+(E) =  j_1 + \phi(S_{i_1}) > \cdots > j_k + \phi(S_{i_k}) = \phi_\sigma^-(E).
\]
Then
\[
T_1(E) \in \{ T_1(S_{i_1})^{\oplus q_1}[j_1],\ldots, T_1(S_{i_k})^{\oplus q_k}[j_k] \}.
\]
Recall from Lemma \ref{lem:sphericaltwist} that $T_1(S_1) = S_1[1-N]$ and $T_1(S_2)\in\{ S_2, S_1\}$. Combining with Lemma \ref{lemma:boundofphase}, we have
\[
\phi_\sigma^+(T_1(E)) \leq \max_{1\leq\ell\leq k} \{j_\ell+\phi(S_{i_\ell})\} = \phi_\sigma^+(E)
\]
and
\[
\phi_\sigma^-(T_1(E)) \geq \min_{1\leq\ell\leq k} \{j_\ell+\phi(S_{i_\ell})\}-(N-1) = \phi_\sigma^-(E) - (N-1).
\]
Similarly, one can apply $T_1^{-1}$ on $E$:
\[
T_1^{-1}(E) \in \{ T_1^{-1}(S_{i_1})^{\oplus q_1}[j_1],\ldots, T_1^{-1}(S_{i_k})^{\oplus q_k}[j_k] \}.
\]
Again by Lemma \ref{lem:sphericaltwist} and \ref{lemma:boundofphase}, we have
\[
\phi_\sigma^+(T_1^{-1}(E)) \leq \max_{1\leq\ell\leq k} \{j_\ell+\phi(S_{i_\ell})\} + (N-1)= \phi_\sigma^+(E) + (N-1)
\]
and
\[
\phi_\sigma^-(T_1^{-1}(E)) \geq \min_{1\leq\ell\leq k} \{j_\ell+\phi(S_{i_\ell})\} = \phi_\sigma^-(E).
\]
By applying $T_1$ to both inequalities, one gets
\[
\phi_\sigma^-(T_1(E)) \leq \phi_\sigma^-(E) \leq \phi_\sigma^+(E) \leq \phi_\sigma^+(T_1(E)) + (N-1).
\]
This proves the inequalities in (1).
The proof of inequalities in (2) is similar.
Hence
\[
\Big|\phi_\sigma^\pm(E)-\phi_\sigma^\pm(T_1(E))-\frac{N-1}{2}\Big| \leq \frac{N-1}{2} \text{ and }
\Big|\phi_\sigma^\pm(E)-\phi_\sigma^\pm(T_2^{-1}(E))+\frac{N-1}{2}\Big| \leq \frac{N-1}{2}
\]
for any nonzero object $E\in\D$.

Let $\Phi=T_1^{a_1}T_2^{-b_1}T_1^{a_2}T_2^{-b_2}\cdots T_1^{a_p}$ for some $a_i,b_i\geq0$.
Define
\[
s\coloneqq\sum_{i=1}^p a_i + \sum_{i=1}^{p-1} b_i \text{ \ and \ } x\coloneqq\sum_{i=1}^{p-1} b_i - \sum_{i=1}^p a_i.
\]
By the above inequalities, we have
\[
\Big|\phi^\pm_\sigma(E) - \phi^\pm_\sigma(\Phi E) + \frac{x(N-1)}{2} \Big| \leq \frac{s(N-1)}{2}
\]
for any nonzero object $E$.
Therefore, under the assumption that $\log\rho(M_\Phi)\geq\frac{N-1}{2}(\sum a_i+\sum b_i)$, we have 
\[
\Big|\phi^\pm_\sigma(E) - \phi^\pm_\sigma(\Phi E) + \frac{x(N-1)}{2} \Big| \leq \log\lambda.
\]
%
Hence, in order to prove statement (a), i.e. that  $\bar{d}(\bar\sigma, \Phi\bar\sigma)\leq\log\lambda$,
it remains to show that
\[
m_\sigma(\Phi E) \leq \lambda m_\sigma(E) \ \ \mathrm{and} \ \ 
m_\sigma(\Phi^{-1} E) \leq \lambda m_\sigma(E)
\]
for any nonzero object $E$.
Suppose the Harder--Narasimhan filtration of $E$ has $x_1$ $S_1$'s and $x_2$ $S_2$'s (up to shifts).
Write $x=(x_1\  x_2)^T$ as a vector.
Then the mass of $E$ is given by $m_\sigma(E)=v^Tx$ since $|Z_\sigma(S_1)|=v_1$ and $|Z_\sigma(S_2)|=v_2$.
By Lemma \ref{lemma:boundofphase}, we have
\[
m_\sigma(\Phi E) \leq v^TM_\Phi x = \lambda v^Tx=\lambda m_\sigma(E).
\]
To show that $m_\sigma(\Phi^{-1} E) \leq \lambda m_\sigma(E)$, recall from Lemma \ref{lem:sphericaltwist} that
\[
T_1^{-1}(S_1) = S_1[N-1] \text{ \ and \ } T_1^{-1}(S_2) \in \{ S_1[N-2], S_2 \},
\]
\[
T_2(S_1)\in \{ S_1, S_2[2-N]\}  \text{ \ and \ }  T_2(S_2) = S_2[1-N].
\]
Hence $T_1^{-1}$ corresponds to the matrix
$
\begin{pmatrix}
1 & 1 \\
0 & 1
\end{pmatrix}
$
in the sense that if
\[
F \in \{ S_{i_1}^{\oplus q_1}[j_1], \ S_{i_2}^{\oplus q_2}[j_2], \ldots, \  S_{i_k}^{\oplus q_k}[j_k] \},
\]
where there are $x_1$ $S_1$'s and $x_2$ $S_2$'s 
among $\{ S_{i_1}^{\oplus q_1}[j_1], \ S_{i_2}^{\oplus q_2}[j_2], \ldots, \  S_{i_k}^{\oplus q_k}[j_k] \}$
(here $S_i^{\oplus q}$ counts as $q$ $S_i$'s),
then $T_1^{-1}(F)$ can be written as
\[
T_1^{-1}(F) \in \{ S_{c_1}^{\oplus p_1}[d_1], \ S_{c_2}^{\oplus p_2}[d_2], \ldots, \  S_{c_\ell}^{\oplus p_\ell}[d_\ell] \},
\]
where there are $x_1+x_2$ $S_1$'s and $x_2$ $S_2$'s among
$\{ S_{c_1}^{\oplus p_1}[d_1], \ S_{c_2}^{\oplus p_2}[d_2], \ldots, \  S_{c_\ell}^{\oplus p_\ell}[d_\ell] \}$.
Similarly,  $T_2$ corresponds to the matrix
$
\begin{pmatrix}
1 & 0 \\
1 & 1
\end{pmatrix}.
$
Hence the corresponding matrix of $\Phi^{-1}=T_1^{-a_p}T_2^{b_{p-1}}\cdots T_2^{b_1}T_1^{-a_1}$ is
\[
M_{\Phi^{-1}} = 
\begin{pmatrix} 
1 & 1 \\
0 & 1 
\end{pmatrix}^{a_p}
\cdots
\begin{pmatrix} 
1 & 0 \\
1 & 1 
\end{pmatrix}^{b_1}
\begin{pmatrix} 
1 & 1 \\
0 & 1 
\end{pmatrix}^{a_1}=M_\Phi
\]
which coincides with $M_\Phi$ since $\Phi$ is assumed to be palindromic.
Hence by Lemma \ref{lemma:boundofphase},
\[
m_\sigma(\Phi^{-1} E) \leq v^TM_{\Phi^{-1}} x = v^TM_\Phi x = \lambda v^Tx=\lambda m_\sigma(E).
\]
This concludes the proof of (a) for the metric space $(\Stab^\dagger(\D(\Gamma_NA_2))/\C, \bar d)$ under the assumption that $\log\rho(M_\Phi)\geq\frac{N-1}{2}(\sum a_i+\sum b_i)$.
The proof for $(\cM_{\Stab^\dagger}(\D(\Gamma_NA_2))/\R^+,\bar d)$ is even simpler because we only need to deal with the mass term, and we do not have to use the extra assumption on $\log\rho(M_\Phi)$.

\medskip

\noindent\emph{Proof of (b):}
We claim that
\[
m_\sigma(\Phi^n S_1) = \lambda^n m_\sigma(S_1) = m_\sigma(\Phi^{-n}S_1)
\]
for any $n\in\N$.
By Lemma \ref{lem:sphericaltwist},
\[
\Phi^n (S_1)\in \{ S_{i_1}^{\oplus q_1}[(N-1)j_1], S_{i_2}^{\oplus q_2}[(N-1)j_2], \ldots, S_{i_k}^{\oplus q_k}[(N-1)j_k] \}
\]
for some $i_1,\ldots,i_k\in\{1,2\}$, $j_1,\ldots,j_k\in\Z$, and $q_1,\cdots, q_k\in\N$.
By Lemma \ref{lem:swapHNfactors}, one can reorder the sequence to get the Harder--Narasimhan filtration of $\Phi^n (S_1)$ since $N\geq3$.
Therefore
\[
m_\sigma(\Phi^n S_1) =  v^TM_\Phi^n \begin{pmatrix}1\\0\end{pmatrix} = \lambda^n v^T\begin{pmatrix}1\\0\end{pmatrix} = \lambda^n m_\sigma(S_1).
\]
On the other hand, again by Lemma \ref{lem:sphericaltwist}, we have
\[
\Phi^{-n} (S_1)\in \{ S_{1}^{\oplus p_1}[(N-1)i_1], S_{2}^{\oplus q_1}[(N-1)j_1+1], \ldots,
S_1^{\oplus p_k}[(N-1)i_k],   S_{2}^{\oplus q_k}[(N-1)j_k+1]  \}
\]
for some $p_1,q_1,\ldots,p_k,q_k\geq0$ and $j_1,\ldots,j_k\in\Z$.
(The degree shifts of $S_1$'s are divisible by $N-1$, while the degree shifts of $S_2$'s are equal to $1$ modulo $N-1$.)
By Lemma \ref{lem:swapHNfactors}, one can reorder the sequence to get the Harder--Narasimhan filtration of $\Phi^{-n} (S_1)$ since $N\geq3$.
Using the assumption that $\Phi$ is palindromic, we have
\[
m_\sigma(\Phi^{-n} S_1) = v^TM_{\Phi^{-1}}^n \begin{pmatrix}1\\0\end{pmatrix}  =  v^TM_\Phi^n \begin{pmatrix}1\\0\end{pmatrix} = \lambda^n v^T\begin{pmatrix}1\\0\end{pmatrix} = \lambda^n m_\sigma(S_1).
\]
This proves the claim that $m_\sigma(\Phi^n S_1) = \lambda^n m_\sigma(S_1) = m_\sigma(\Phi^{-n}S_1)$.
Therefore
\begin{align*}
\bar d(\bar\sigma, \Phi^n\bar\sigma) & \geq \inf_{y>0} \sup_{E\neq0}
\Big\{
|\log\frac{ym_\sigma(\Phi^n E)}{m_\sigma(E)}|
\Big\} \\
 & \geq 
 \inf_{y>0} \max
 \Big\{
|\log\frac{ym_\sigma(\Phi^n S_1)}{m_\sigma(S_1)}|,
|\log\frac{ym_\sigma(S_1)}{m_\sigma(\Phi^{-n} S_1)}|
\Big\}
= n\log\lambda.
\end{align*}
Hence we have
$\lim_{n\ra\infty}\frac{1}{n}\bar{d}(\bar\sigma, \Phi^n\bar\sigma) \geq \log\lambda$.
\end{proof}

\subsection{An example from quintic Calabi--Yau threefolds} \label{eg:quintic}
We prove that the autoequivalence on the derived category of quintic Calabi--Yau threefolds $X$ considered by the first author in \cite{Fan1} and Ouchi in \cite{Ouchi} is weak pseudo-Anosov.
The existence of Bridgeland stability conditions on $\D^b(X)$ has been established by Li \cite{Li} recently.
There is a distinguished connected component of $\Stab(\D^b(X))$ containing \emph{geometric stability conditions} for which skyscraper sheaves are stable and of the same phase.
We choose $\Stab^\dagger(\D)$ in Assumption (B) to be the distinguished connected component.

\begin{Thm} \label{Thm:pA_quintic}
Let $X$ be a quintic Calabi--Yau hypersurface in $\C\P^4$.
Then
\[
\Phi\coloneqq T_\O\circ(-\otimes\O(-1))
\]
is a weak pseudo-Anosov autoequivalence on $\D^b(X)$.
Here $T_\O$ denotes the spherical twist with respect to the structure sheaf $\O_X$.
\end{Thm}

\begin{proof}
Take any $\sigma\in\Stab^\dagger(\D^b(X))$.
Let $G\in\D^b(X)$ be a split generator and $E\in\D^b(X)$ be any nonzero object.
By Theorem \ref{Thm:Ikeda}, we have
\[
\limsup_{n\ra\infty} \frac{1}{n} \log |Z_\sigma(\Phi^nE)| \leq
\limsup_{n\ra\infty} \frac{1}{n} \log m_\sigma(\Phi^nE) \leq
\limsup_{n\ra\infty}\frac{1}{n}\log m_\sigma(\Phi^n G) \leq h_\cat(\Phi).
\]
Here $h_\cat(\Phi)$ denotes the categorical entropy of $\Phi\in\Aut(\D^b(X))$.
By \cite[Remark 4.3]{Fan1}, we have
\[
  h_\cat(\Phi) = \log\rho([\Phi]).
\]

Let $\lambda_1, \lambda_2, \lambda_3, \lambda_4$ be the eigenvalues of $[\Phi]\in\Aut(H^\mathrm{ev}(X;\C))$
in which $\lambda_1=\rho([\Phi])>1$, and let $v_1,v_2,v_3,v_4\in H^\mathrm{ev}(X;\C)$ be the corresponding eigenvectors.
By \cite{Bri}, one can deform the stability condition $\sigma$ if needed so that $Z_\sigma(v_1)\neq0$. Then we have 
\[
\limsup_{n\ra\infty} \frac{1}{n} \log |Z_\sigma(\Phi^nE)| = \log\rho([\Phi]) \  \Longleftrightarrow \ 
[E]\notin \mathrm{Span}\{ v_2, v_3, v_4\}.
\]
One can check by brute force computations that the span of $v_2,v_3,v_4$ does not contain any nonzero rational vectors, see
 Appendix \ref{CY3}.
This proves that the mass growth
\[
\limsup_{n\ra\infty} \frac{1}{n} \log m_\sigma(\Phi^nE) = \log \lambda_1 \approx 2.04 >0
\]
for any object $E$ with $[E]\neq0$ is the same positive number.
Hence $\Phi=T_\O\circ(-\otimes\O(-1))$ is weak pseudo-Anosov.
\end{proof}

Note that the same argument can be used to show that the autoequivalence $\Phi=T_\O\circ(-\otimes\O(-1))$ is weak pseudo-Anosov on $\D^b(X)$ for any Calabi--Yau manifold $X$ of odd dimension,
assuming the existence of Bridgeland stability conditions on $\D^b(X)$,
and assuming that there is no integral vector in the span of the eigenvectors of eigenvalues other than the largest one.

\subsection{Other examples}

We show that on the derived category of curves,
the notion of pseudo-Anosov autoequivalences coincides with DHKK pseudo-Anosov autoequivalences.

\begin{Prop}\label{Prop:curve}
Let $\D=\D^b(C)$ be the derived category of a smooth projective curve $C$ over $\C$.
Then $\Phi\in\Aut(\D)$ is pseudo-Anosov if and only if it is DHKK pseudo-Anosov.
\end{Prop}

\begin{proof}
By Kikuta \cite[Proposition 4.13, 4.14]{Kikuta}, in the case when $\D$ is the derived category of a curve, $\Phi\in\Aut(\D)$ is DHKK pseudo-Anosov if and only if its categorical entropy $h_\cat(\Phi)>0$.
On other hand, by Theorem \ref{Thm:Ikeda}
an autoequivalence $\Phi$ is pseudo-Anosov implies $h_\cat(\Phi)>0$.
The proposition then follows from Theorem \ref{Thm:more_general}.
\end{proof}


\section{Pseudo-Anosov maps as pseudo-Anosov autoequivalences}
\label{sec_padiff}

In this section we discuss two ways in which a pseudo-Anosov map can induce a pseudo-Anosov autoequivalence on the Fukaya category of the surface.
The first is based on the $\mathbb Z$-graded Fukaya category and stability conditions constructed from quadratic differentials, while the second is based on the $\mathbb Z/2$-graded Fukaya category and the mass function defined as complexity with respect to a generator. 
In the second case we cannot quite prove that the induced autoequivalence is pseudo-Anosov, but provide strong evidence for this.
Throughout, we assume that surfaces are oriented and maps orientation preserving.

\subsection{$\mathbb Z$-graded Fukaya category}

Let $S$ be a closed surface with a finite set $M\subset S$ of marked points and suppose
$f\colon S\setminus M\to S\setminus M$ is pseudo-Anosov. 
By definition this means that there is a complex structure on $S$, a meromorphic quadratic differential $\varphi$ on $S$ (not identically zero) with at worst simple poles at each point $p\in M$ and holomorphic on $S\setminus M$, and a stretch factor $\lambda>1$ such that $f$ acts on $\varphi$ by scaling $\mathrm{Re}(\sqrt{\varphi})$ by $\lambda$ and $\mathrm{Im}(\sqrt{\varphi})$ by $1/\lambda$.
Thus, if $F^u$ is the horizontal measured foliation of $\varphi$ and $F^s$ is the vertical measured foliation of $\varphi$, then $F^u$ and $F^s$ are a pair of transverse measured foliations such that the underlying foliations are preserved by $f$ and the measure for $F^u$ (resp. $F^s$) is scaled by a factor $\lambda$ (resp. $1/\lambda$) under the action of $f$.
The definition in terms of a quadratic differential is particularly natural from Bers' point of view of the Thurston classification~\cite{bers78}.

Let $Z\subset S$ be the set of zeros and poles of $\varphi$.
Note that $M\subset Z$ and $Z$ is non-empty unless $S$ is the torus and $\varphi$ is constant.
In order to avoid having to deal with this special case, we assume $Z\neq\emptyset$ from now on, though this is not essential.
Fix an arbitrary ground field $\mathbf k$, then there is a triangulated $A_\infty$-category $\mathcal F=\mathcal F(S\setminus Z,F^u,\mathbf k)$ over $\mathbf k$, the $\mathbb Z$-graded \textit{wrapped Fukaya category} of the punctured surface $S\setminus Z$ with grading foliation $F^u$ and coefficients in $\mathbf k$, see e.g. \cite{HKK} for the construction.
Moreover, according to the main theorem in \cite{HKK}, the quadratic differential $\varphi$ gives a stability condition $\sigma$ on $\mathcal F$ such that stable objects correspond to saddle connections and closed geodesic loops (with grading and $\mathbf k$-linear local system) on the flat surface $(S,|\varphi|)$.
The central charge is given by 
\[
Z([\gamma])=\int_\gamma\sqrt{\varphi}
\]
which is well-defined for 
\[
[\gamma]\in H_1(S,Z;\mathbb Z\sqrt{\varphi})\cong K_0(\mathcal F)
\]
where $\mathbb Z\sqrt{\varphi}$ denotes the local system of integer multiples of choices of $\sqrt{\varphi}$. 

Since the pseudo-Anosov map $f$ preserves the subset $Z$ and the grading foliation $F^u$, it induces an autoequivalence $f_*$ of $\mathcal F$.
From the description of the action of $f$ on $\varphi$ and the correspondence between quadratic differentials and stability conditions is clear that $f$ acts on $\sigma$ like the element $\mathrm{diag}(\lambda,1/\lambda)\in \widetilde{GL^+(2,\mathbb R)}$, i.e. $f_*$ is DHKK pseudo-Anosov.
This example was the motivation for the definition of pseudo-Anosov autoequivalence in \cite{DHKK}.

\subsection{$\mathbb Z/2$-graded Fukaya category}

The construction above has the perhaps unwanted feature that the category $\mc F$ depends on the choice of $f$ in the mapping class group.
Another approach is to consider a version of the Fukaya category of $S$ which does not depend on a grading foliation and is thus only $\mathbb Z/2$-graded.
Such categories do not admit stability conditions, so we need to use the more general notion of pseudo-Anosov autoequivalence introduced here where the mass function on the triangulated category is the complexity with respect to a generator.

Let $S$ be a closed surface of genus $g>1$ and choose a symplectic form $\omega$ on $S$.
The Fukaya category $\mathcal F=\mathcal F(S,\omega)$ of the surface $S$ is an $A_\infty$-category over the Novikov field $\mathbf k((t^{\mathbb R}))$ where $\mathbf k$ is some coefficient field of characteristic zero.
Objects of $\mathcal F$ can be taken to be immersed oriented curves $L$ which do not bound an immersed one-gon (teardrop) together with a unitary rank one local system over $\mathbf k((t^{\mathbb R}))$. 
A self-contained elementary definition of $\mathcal F$ can be found in \cite{abouzaid08}.
We assume that $\mathcal F$ has been formally completed so that its homotopy category $H^0(\mathcal F)$ is triangulated and all idempotents split.

The mapping class group $\mathrm{MCG}(S)$ of $S$ does not act naturally on $\mathcal F$, at least not without making additional choices.
Instead one considers the \textit{symplectic mapping class group} of $S$ which is the quotient $\mathrm{Symp}(S)/\mathrm{Ham}(S)$ of the group of symplectomorphisms by the normal subgroup of hamiltonian symplectomorphisms.
This group surjects onto $\mathrm{MCG}(S)$, i.e. every element of $\mathrm{MCG}(S)$ is realized by some symplectomorphism.

While $\mc F$ does not admit a stability condition, we can still use mass functions of the form $\mu(E)\coloneqq\delta(G,E)$ where $G$ is a generator, e.g. a direct sum of certain $2g$ simple closed curves, and $\delta(G,E)$ denotes the complexity of $E$ with respect to $G$ as in \cite{DHKK}.
These mass functions are all equivalent up to some constant and moreover equivalent to mass functions of the form 
\[
\mu(E)\coloneqq\dim \mathrm{Ext}^{\bullet}(G,E)=\dim\mathrm{Ext}^0(G,E)+\dim\mathrm{Ext}^1(G,E)
\]
since $\mc F$ is smooth and proper.

\begin{Conj}
Let $f$ be a pseudo-Anosov symplectomorphism of a closed symplectic surface $(S,\omega)$.
Then the induced autoequivalence $\Phi=f_*$ of the Fukaya category $\mc F(S,\omega)$ is pseudo-Anosov with respect to the mass function $\mu(E)\coloneqq\dim \mathrm{Ext}^{\bullet}(G,E)$ where $G$ is any generator of $\mc F(S,\omega)$.
Moreover, the stretch factor $\lambda>1$ of $f$ coincides with the stretch factor of the autoequivalence $f_*$.
\end{Conj}

As in \cite{DHKK} one uses the fact from pseudo-Anosov theory that geometric intersection numbers $i(S_1,f^nS_2)$ grow exponentially with rate $\lambda$ to conclude that $\dim \mathrm{Ext}^{\bullet}(G,f_*^nE)$ has exponential growth rate $\lambda$, i.e. $h_{\mu,\Phi}(E)=\log\lambda$, for any object $E\in\mc F$ which comes from a simple closed curve.
Here we take $G$ to be a generator which is a direct sum of simple closed curves.
Since simple closed curves split-generate, we also know that $h_{\mu,\Phi}(E)\leq\log\lambda$ for any $E\in\mathcal F$.
The statement of the conjecture means precisely that equality holds for all non-zero objects $E$ of $\mathcal F$.

The above discussion generalizes, with some changes, to the case of punctured surfaces $S$.
In this case the symplectic form should be chosen to have infinite area, and in fact the definition of the (wrapped) Fukaya category from \cite{HKK} can be used, except for replacing $\mathbb Z$-gradings by $\mathbb Z/2$-gradings.
The category is smooth, but not proper, and the classification of indecomposable objects in terms of immersed curves from \cite{HKK} could be of use.

\section{Open questions}
\label{sec:openquestions}

We propose several open questions related to the study of pseudo-Anosov autoequivalences.

\subsection*{Alternative definitions of pseudo-Anosov autoequivalences}
In Section \ref{sec:section2def} and \ref{sec:examples}, we study pseudo-Anosov autoequivalences with respect to the mass functions given by the Bridgeland stability conditions in a fixed connected component $\Stab^\dagger(\D)\subset\Stab(\D)$.
One can use other mass functions, for instance the mass function $\delta_G$ given by the complexity with respect to a split generator $G$ (Example \ref{eg:complexity}) to define the mass growth of an object in a triangulated category, thereby defines another notion of ``pseudo-Anosov autoequivalence" (with respect to the complexity function) following the same idea as in Section \ref{pA:auto}.

\begin{Def}
\label{def:pAcomplexity}
Let $\D$ be a triangulated category with a split generator $G$.
We say an autoequivalence $\Phi\in\Aut(\D)$ is \emph{pseudo-Anosov with respect to the complexity function} if there exists $\lambda>1$ such that for any nonzero object $E$ in $D$,
\[
\limsup_{n\ra\infty}\frac{\log\delta_G(\Phi^nE)}{n} = \log\lambda>0.
\]
\end{Def}

Note that the mass growth with respect to the complexity function is independent of the choice of the split generator: if $G'$ is another split generator of $\D$, by \cite[Proposition~2.3]{DHKK}, we have
$\delta_G(\Phi^nE)\leq\delta_G(G')\delta_{G'}(\Phi^nE)$.
One can also use the mass function $\mu_G$ given by the dimensions of $\Hom^\bullet(G,-)$ in Example \ref{eg:Hom} to define yet another notion of pseudo-Anosov autoequivalences.

\begin{Def}
\label{def:pAExtdistance}
Let $\D$ be a triangulated category with a split generator $G$.
We say an autoequivalence $\Phi\in\Aut(\D)$ is \emph{pseudo-Anosov with respect to the Ext-distance function} if there exists $\lambda>1$ such that for any nonzero object $E$ in $D$,
\[
\limsup_{n\ra\infty}\frac{\log\mu_G(\Phi^nE)}{n} =
\limsup_{n\ra\infty}\frac{\log\sum_{i\in\Z}\dim\Hom_\D(G, E[i])}{n}=
\log\lambda>0.
\]
\end{Def}

It turns out that Definition \ref{def:pAcomplexity} and \ref{def:pAExtdistance} are equivalent, since there exist constants $C_1,C_2>0$ depending on $G$ such that
$C_1\delta_G(E) \leq \mu_G(E) \leq C_2\delta_G(E)$
holds for any nonzero object $E$ \cite[proof of Theorem~2.7]{DHKK}.
It is then important to determine whether the definition via complexity function (or equivalently, Ext-distance function) coincides with our previous definition via stability conditions. More precisely:

\begin{Ques}
Let $\D$ be a triangulated category with a split generator and $\Stab(\D)\neq\emptyset$.
    \begin{itemize}
        \item Is it true that the property of being a pseudo-Anosov autoequivalence does not depend on the choice of a connected component of $\Stab(\D)$?
        \item Does the notion of ``pseudo-Anosov autoequivalence" with respect to the complexity function (or equivalently the Ext-distance function) coincide with our definition of pseudo-Anosov autoequivalence (Definition \ref{Def:pAauto})?
    \end{itemize}
\end{Ques}

We prove in Appendix \ref{app:ellcurve} that both questions have affirmative answers if $\D$ is the bounded derived category of coherent sheaves on an elliptic curve (Theorem \ref{Thm:massequivalentEllcurve}). Note that it is proved by Ikeda \cite[Theorem~3.14]{Ikeda} that if a connected component $\Stab^\circ(\D)\subset\Stab(\D)$ contains an algebraic stability condition, then the mass growths with respect to $m_\sigma$ and $\delta_G$ coincide for any $\sigma\in\Stab^\circ(\D)$.
Hence both questions above also have affirmative answers for triangulated categories $\D$ with the property that each connected component of $\Stab(\D)$ contains an algebraic stability condition, e.g.~$\D^b\Coh(\P^1)$ \cite{DK1,MacriCurve,Okada}.
Examples of triangulated categories admitting algebraic stability conditions include derived categories with full strong exceptional collections, see for instance \cite[Remark 4.8, 4.9, 4.10]{DK2}, \cite{MacriExceptional}, \cite{QW}.

\subsection*{Genericity of pseudo-Anosov autoequivalences}
\begin{Ques}
Let $\D$ be a Calabi--Yau category satisfying Assumptions (A) and (B).
Are pseudo-Anosov autoequivalences ``generic" in $\Aut(\D)$?
\end{Ques}

Note that the statement is not true if one removes the Calabi--Yau condition.
For instance, the derived category $\D^b(C)$ of coherent sheaves on a smooth projective curve $C$ with genus $g(C)\neq1$ does not have any pseudo-Anosov autoequivalences by Kikuta \cite[Proposition 4.13]{Kikuta}
and Proposition \ref{Prop:curve}.

\subsection*{Irreducible autoequivalences}
In light of the Nielsen--Thurston classification of mapping class group elements, it is natural to consider the following definition.
\begin{Def} \label{Def:irred}
Let $\Phi\colon\D\ra\D$ be an autoequivalence of a triangulated category $\D$.
We say $\Phi$ is
    \begin{itemize}
        \item \emph{irreducible} if it has no $\Phi$-invariant proper thick triangulated subcategory;
        \item \emph{strongly irreducible} if there is no finite collection of proper thick triangulated subcategories that are permuted under $\Phi$.
    \end{itemize}
\end{Def}
It follows from the definition that if an autoequivalence $\Phi$ is irreducible and $h_{m_\sigma,\Phi}(E)>1$ for some object $E\in\D$ and $\sigma\in\Stab^\dagger(\D)$, then $\Phi$ is pseudo-Anosov.
Indeed, if $\Phi$ is irreducible, then the associated growth filtration of $\Phi$ can only have one step:
\[
0\subset\D_\lambda=\D.
\]
The condition that there exists an object with positive mass growth $h_{m_\sigma,\Phi}(E)>1$ implies $\lambda>1$.
Hence $\Phi$ is pseudo-Anosov.
Similarly, if an autoequivalence $\Phi$ is irreducible and has positive categorical entropy (Definition~\ref{def:catentropy}), then $\Phi$ is pseudo-Anosov with respect to the complexity function (Definition~\ref{def:pAcomplexity}).
An interesting question is whether the reverse implication holds.

\begin{Ques}
Are (DHKK) pseudo-Anosov autoequivalences (strongly) irreducible?
\end{Ques}
\begin{Ques}
Do pseudo-Anosov autoequivalences or irreducible autoequivalences with positive entropy on a triangulated category $\D$
act hyperbolically on the metric space $\Stab^\dagger(\D)/\C$ or $\cM_{\Stab^\dagger}(\D)/\R_{>0}$?
\end{Ques}

\subsection*{Other examples of pseudo-Anosov autoequivalences}
It would be interesting to find pseudo-Anosov autoequivalences on triangulated categories that are not Calabi--Yau as well. For instance, it is proved in \cite[Theorem 2.17]{DHKK} that the Serre functor on the derived category of quiver representations has positive entropy if the quiver is not of extended Dynkin type. Moreover, the Serre functor on the derived category of representations of Kronecker quiver with at least three arrows is DHKK pseudo-Anosov.
\begin{Ques}
Is the Serre functor on the derived category of quiver representations pseudo-Anosov
if the quiver is not of extended Dynkin type?
\end{Ques}

\appendix

\section{The case of elliptic curves}
\label{app:ellcurve}
Let $\D=\D^b\Coh(X)$ be the bounded derived category of coherent sheaves on an elliptic curve $X$.
We prove in this section that the notion of pseudo-Anosov autoequivalences with respect to Bridgeland stability conditions (Definition \ref{Def:pAauto}), complexity function (Definition \ref{def:pAcomplexity}), and Ext-distance function (Definition \ref{def:pAExtdistance}) are all equivalent.

Recall that $\Stab(\D)\cong\C\times\H$ is connected \cite{Bri}, hence the mass growth rates with respect to any Bridgeland stability condition on $\D$ are the same (Proposition~\ref{prop:sameconnectedcomponent}). We choose the stability condition $\sigma\in\Stab(\D)$ such that
\begin{itemize}
    \item $Z_\sigma(E)=-\deg(E)+i\cdot\mathrm{rank}(E)$.
    \item For $0<\phi\leq1$, the (semi)stable objects $P_\sigma(\phi)$ are the slope-(semi)stable coherent sheaves whose central charges lie in the ray $\R_{>0}\cdot\exp(i\pi\phi)$.
    \item For other $\phi\in\R$, define $P_\sigma(\phi)$ via the property $P_\sigma(\phi+1)=P_\sigma(\phi)[1]$.
\end{itemize}

\begin{Prop}
\label{prop:ellcurveHommassestimate}
Fix an ample line bundle $\O(1)$ on the elliptic curve $X$ and a split generator $G=\O\oplus\O(1)$ of $\D$. There exists a constant $C>0$ such that
\[
\mu_G(E)\coloneqq\sum_{i\in\Z}\dim\Hom_\D(G, E[i]) < C\cdot m_\sigma(E)
\]
holds for any object $E$ in $\D$.
\end{Prop}

\begin{proof}
By the triangle inequality of $\mu_G$ (cf.~Definition~\ref{Def:Massfn}(1)), it suffices to show the inequality holds for slope-semistable coherent sheaves.
By considering Jordan--Holder filtrations of semistable coherent sheaves, one can further reduce to proving the inequality only for stable vector bundles and indecomposable torsion sheaves.

Let $E$ be an indecomposable torsion sheaf on $X$. Then $E\cong\O_X/\O_X(-sx_0)$ for some $x_0\in X$ and $s>0$. We have
\begin{align*}
    \mu_G(E) & =\dim H^0(E)+\dim H^1(E)+\dim H^0(E(-1))+\dim H^1(E(-1)) \\
    & =2s = 2m_\sigma(E).
\end{align*}

Let $E$ be a stable vector bundle on the elliptic curve $X$ of rank $r$ and degree $d$. By \cite[Lemma~15]{Atiyah} and Hirzebruch--Riemann--Roch theorem,
    \begin{itemize}
        \item if $d>0$, then $\dim H^0(E)=d$ and $\dim H^1(E)=0$;
        \item if $d=0$, then $\dim H^0(E)=\dim H^1(E)$ and both equal to $0$ or $1$;
        \item if $d<0$, then $\dim H^0(E)=0$ and $\dim H^1(E)=-d$.
    \end{itemize}
Hence we have $\dim H^0(E) + \dim H^1(E) \leq |d|+2$ for any stable vector bundle $E$.
Thus
\begin{align*}
    \mu_G(E) & =\dim H^0(E)+\dim H^1(E)+\dim H^0(E(-1))+\dim H^1(E(-1)) \\
    & \leq |d| + |d-r| + 4 \\
    & < \sqrt{r^2+d^2} + \sqrt{2(r^2+d^2)} + 4\sqrt{r^2+d^2} \\
    & = (5+\sqrt2)m_\sigma(E).
\end{align*}
\end{proof}

\begin{Thm}
\label{Thm:massequivalentEllcurve}
Let $\D=\D^b\Coh(X)$ be the bounded derived category of coherent sheaves on an elliptic curve $X$. Let $G=\O\oplus\O(1)$ be a split generator of $\D$. Then
\[
\limsup_{n\ra\infty}\frac{\log m_\sigma(\Phi^nE)}{n}=
\limsup_{n\ra\infty}\frac{\log \delta_G(\Phi^nE)}{n}=
\limsup_{n\ra\infty}\frac{\log \mu_G(\Phi^nE)}{n}
\]
holds for any nonzero object $E$ in $\D$.
Hence the notion of pseudo-Anosov autoequivalences with respect to these three mass functions are equivalent.
\end{Thm}

\begin{proof}
By Ikeda \cite[Proposition~3.4]{Ikeda},  we have $m_\sigma(E)\leq m_\sigma(G)\delta_G(E)$ for any nonzero object $E$ in $\D$.
Combining with \cite[proof of Theorem~2.7]{DHKK} (cf.~Section~\ref{sec:openquestions}) and Proposition \ref{prop:ellcurveHommassestimate}, we obtain
\[
\limsup_{n\ra\infty}\frac{\log \delta_G(\Phi^nE)}{n}=
\limsup_{n\ra\infty}\frac{\log \mu_G(\Phi^nE)}{n}\leq
\limsup_{n\ra\infty}\frac{\log m_\sigma(\Phi^nE)}{n}\leq
\limsup_{n\ra\infty}\frac{\log \delta_G(\Phi^nE)}{n}
\]
for any nonzero object $E$.
This concludes the proof.
\end{proof}

\section{Basic properties of mass functions} \label{Append:Mass}
We defined the notion of mass functions on triangulated categories in Definition \ref{Def:Massfn}.
A nice property that mass functions have but Bridgeland stability conditions do not have is that they are functorial.

\begin{Def}
Let $F:\cC\ra\D$ be an exact functor between triangulated categories.
Let $\mu:\Ob(\D)\ra\R_{\geq0}$ be a mass function on $\D$.
Then one can define the pullback of the mass function via $F$:
\[
F^*\mu: \Ob(\cC) \ra \R_{\geq0}, \ \ \ E \mapsto \mu(F(E)),
\]
which is clearly a mass function on $\cC$.
\end{Def}

Another nice property of mass functions is that one can induce relative mass functions from mass functions.

\begin{Def}
Let $\D'$ be a thick triangulated subcategory of $\D$.
Let $\mu:\Ob(\D)\ra\R_{\geq0}$ be a mass function on $\D$.
Define the induced mass function $i_*^{\D'}\mu: \Ob(\D)\ra\R_{\geq0}$ by
$$
(i_*^{\D'}\mu)(E):=
\inf\left\{ \displaystyle
\sum_{A_i\notin\D'} \mu(A_i)
 \,\middle|\,
\begin{xy}
(0,5) *{0}="0", (20,5)*{\star}="1", (30,5)*{\dots}, (40,5)*{\star}="k-1", 
(60,5)*{E\oplus \star}="k",
(10,-5)*{A_1}="n1", (30,-5)*{\dots}, (50,-5)*{A_n}="nk",
\ar "0"; "1"
\ar "1"; "n1"
\ar@{-->} "n1";"0" 
\ar "k-1"; "k" 
\ar "k"; "nk"
\ar@{-->} "nk";"k-1"
\end{xy}
\, \right\}.
$$
\end{Def}

While it is not clear to us how to pushforward a mass function from $\D$ to a quotient $\D/\D'$,
the mass function $i_*^{\D'}\mu$ can serve as a replacement of this process.
In particular, one can see that the induced mass function $i_*^{\D'}\mu$ vanishes on $\D'$.
The proposition below shows the compatibility between the pullback of mass functions via quotient maps
and the induced relative mass functions.

\begin{Prop}
Let $\D'$ be a thick triangulated subcategory of $\D$,
and $q_{D'}: \D\ra \D/\D'$ be the quotient map.
Let $\mu:\Ob(\D/\D')\ra\R_{\geq0}$ be a mass function on $\D/\D'$.
Then
\[
i_*^{\D'} q_{\D'}^* \mu = q_{\D'}^* \mu.
\]
\end{Prop}

\begin{proof}
Since $q^*_{\D'}\mu$ vanishes on $\D'$, for any sequence of exact triangles 
\[
\xymatrix@C=.5em{
& 0 \ar[rrrr] &&&& \star \ar[rrrr] \ar[dll] &&&& \star
\ar[rr] \ar[dll] && \ldots \ar[rr] && \star
\ar[rrrr] &&&& E\oplus\star, \ar[dll]  &  \\
&&& A_1 \ar@{-->}[ull] &&&& A_2 \ar@{-->}[ull] &&&&&&&& A_n \ar@{-->}[ull] 
}
\]
we have
\[
(q_{\D'}^* \mu)(E) \leq \sum_i (q_{\D'}^* \mu)(A_i) = \sum_{A_i\notin\D'}(q_{\D'}^* \mu)(A_i).
\]
\end{proof}

\section{Computation in the quintic Calabi--Yau threefold case} \label{CY3}

We use the notations in Section \ref{eg:quintic}.
Let $H=\O(1)$ be the ample generator of $\Pic(X)$.
The powers $1,H,H^2,H^3$ form a basis of $H^{\mathrm{ev}}(X;\Q)$.
With respect to this basis, $[\Phi]$ corresponds to the following matrix
\[\left(\begin{matrix}  
6 & -\frac{20}{3} & 5 &-5\\
-1 &1&0&0\\
\frac{1}{2} &-1&1&0\\
-\frac{1}{6} & \frac{1}{2} &-1&1
\end{matrix}\right).\]
Its eigenvalues and corresponding eigenvectors are:
\begin{align*}
\lambda_1 &= \frac{1}{4} \left(   9+3 \sqrt{5} + \sqrt{110+54 \sqrt{5}}   \right) \\
\lambda_2 &= \frac{1}{4} \left(   9-3 \sqrt{5} + i\sqrt{-110+54 \sqrt{5}}   \right) \\
\lambda_3 &= \frac{1}{4} \left(   9-3 \sqrt{5} -i \sqrt{-110+54 \sqrt{5}}   \right) \\
\lambda_4 &=\frac{1}{4} \left(   9+3 \sqrt{5} - \sqrt{110+54 \sqrt{5}}   \right)
\end{align*}
and
\[v_1=\left(\begin{matrix} 
\frac{-3}{122} \left(  185 + 75 \sqrt{5} + \sqrt{ 104150 + 47070 \sqrt{5}}   \right) \\
\frac{1}{61}(60+54 \sqrt{5})\\
\frac{-3}{61}   \sqrt{ 575+426 \sqrt{5}}  \\
1
\end{matrix} \right),
v_2=\left(\begin{matrix} 
\frac{3}{122} \left(  -185 + 75 \sqrt{5} + i \sqrt{ -104150 + 47070 \sqrt{5}}   \right) \\
\frac{1}{61}(60-54 \sqrt{5})\\
\frac{-3}{61} i   \sqrt{ -575+426 \sqrt{5}}   \\
1
\end{matrix} \right)\]
\[v_3=\left(\begin{matrix} 
\frac{3}{122} \left(  -185 + 75 \sqrt{5} - i \sqrt{ -104150 + 47070 \sqrt{5}}   \right) \\
\frac{1}{61}(60-54 \sqrt{5})\\
\frac{3}{61} i   \sqrt{ -575+426 \sqrt{5}}\\
1
\end{matrix} \right), 
v_4=\left(\begin{matrix} 
\frac{-3}{122} \left(  185 + 75 \sqrt{5} - \sqrt{ 104150 + 47070 \sqrt{5}}   \right)  \\
\frac{1}{61}(60+54 \sqrt{5})\\
\frac{3}{61}   \sqrt{ 575+426 \sqrt{5}}\\
1
\end{matrix} \right)\]

The following lemma follows from straightforward but tedious computations.

\begin{Lem}
The span of $\{v_2,v_3,v_4\}$ over complex numbers $\C$ does not contain any rational vector.
\end{Lem}

%
%
%
%
%
%
%
%
%



\begin{thebibliography}{99}

\bibitem{abouzaid08} M.~Abouzaid,
\emph{On the Fukaya categories of higher genus surfaces},
Adv. Math. 217 (2008), no. 3, 1192-1235.

\bibitem{Amiot}C.~Amiot,
\emph{Cluster categories for algebras of global dimension 2 and quivers with potential},
Ann. Inst. Fourier (Grenoble) 59 (2009), no. 6, 2525-2590.

\bibitem{AB13}D.~Arcara and A.~Bertram,
\emph{Bridgeland-stable moduli spaces for K-trivial surfaces},
J. Eur. Math. Soc. (JEMS), 15 (2013), no. 1, 1-38.
With an appendix by Max Lieblich.

\bibitem{Atiyah}M.~F.~Atiyah,
\emph{Vector bundles over an elliptic curve},
Proc. London Math. Soc., 3 (1957), no. 7, 414-452.

\bibitem{BMS16}A.~Bayer, E.~Macr\`i, and P.~Stellari,
\emph{The space of stability conditions on abelian threefolds, and on some Calabi--Yau threefolds},
Invent. Math., 206 (2016), no. 3, 869-933.

\bibitem{BMT14}A.~Bayer, E.~Macr\`i, and Y.~Toda,
\emph{Bridgeland stability conditions on threefolds I: Bogomolov--Gieseker type inequalities},
J. Algebraic Geom., 23 (2014), no. 1, 117-163.

\bibitem{BMSZ17}M.~Bernardara, E.~Macr\`i, B.~Schmidt, and X.~Zhao,
\emph{Bridgeland stability conditions on Fano threefolds},
\'Epijournal Geom. Alg\'ebrique, 1:Art. 2, 24, 2017.

\bibitem{bers78} L. Bers,
\emph{An extremal problem for quasiconformal mappings
and a theorem by Thurston},
Acta Math. 141 (1978), no. 1-2, 73-98.

\bibitem{Bri}T.~Bridgeland,
\emph{Stability conditions on triangulated categories},
Ann. of Math. (2) 166 (2007), no. 2, 317-345.

\bibitem{BriK3}T.~Bridgeland,
\emph{Stability conditions on K3 surfaces},
Duke Math. J., 141 (2008), no. 2, 241-291.

\bibitem{Bri09}T.~Bridgeland,
\emph{Spaces of stability conditions},
Algebraic geometry—Seattle 2005. Part 1, 1-21,
Proc. Sympos. Pure Math., 80, Part 1, Amer. Math. Soc., Providence, RI, 2009.

\bibitem{BQS}T.~Bridgeland, Y.~Qiu and T.~Sutherland,
\emph{Stability conditions and the $A_2$ quiver},
Adv. Math. 365 (2020), 107049, 33 pp.

\bibitem{BS}T.~Bridgeland and I.~Smith,
\emph{Quadratic differentials as stability conditions},
Publ. Math. Inst. Hautes Études Sci. 121 (2015), 155-278. 

\bibitem{DHKK}G.~Dimitrov, F.~Haiden, L.~Katzarkov, and M.~Kontsevich,
\emph{Dynamical systems and categories},
Contemp. Math. 621 (2014), 133-170.

\bibitem{DK1}G.~Dimitrov and L.~Katzarkov,
\emph{Bridgeland stability conditions on wild Kronecker quivers},
Adv. Math. 352 (2019), 27-55.

\bibitem{DK2}G.~Dimitrov and L.~Katzarkov,
\emph{Some new categorical invariants},
Selecta Math. (N.S.) 25 (2019), no. 3, Paper No. 45, 60 pp.

\bibitem{Fan1}Y.-W.~Fan,
\emph{Entropy of an autoequivalence on Calabi-Yau manifolds},
Math. Res. Lett. 25 (2018), no. 2, 509-519. 

\bibitem{Fan2}Y.-W.~Fan,
\emph{Systoles, Special Lagrangians, and Bridgeland stability conditions},
arXiv:1803.09684.

\bibitem{GMN}D.~Gaiotto, G.~W.~Moore and A.~Neitzke,
\emph{Wall-crossing, Hitchin systems, and the WKB approximation},
Adv. Math. 234 (2013), 239-403.

\bibitem{Gin}V.~Ginzburg,
\emph{Calabi--Yau algebras},
arXiv:math/0612139.

\bibitem{Hai}F.~Haiden,
\emph{An extension of the Siegel space of complex abelian varieties and conjectures on stability structures},
Manuscripta Math. 163 (2020), no. 1-2, 87-111.

\bibitem{HKK}F.~Haiden, L.~Katzarkov and M.~Kontsevich,
\emph{Flat surfaces and stability structures},
Publ. Math. Inst. Hautes Études Sci. 126 (2017), 247-318.

\bibitem{Ikeda2}A.~Ikeda,
\emph{Stability conditions on CYN categories associated to $A_n$-quivers and period maps},
Math. Ann. 367 (2017), no. 1-2, 1-49.

\bibitem{Ikeda}A.~Ikeda,
\emph{Mass growth of objects and categorical entropy},
Nagoya Math. J. (2020), 1-22.

\bibitem{Keller}B.~Keller,
\emph{Deformed Calabi-Yau completions},
J. Reine Angew. Math., 654 (2011), 125-180.
With an appendix by Michel Van den Bergh.

\bibitem{KellerSurvey}B.~Keller,
\emph{Cluster algebras and derived categories},
arXiv:1202.4161.

\bibitem{KellerYang}B.~Keller and D.~Yang,
\emph{Derived equivalences from mutations of quivers with potential},
Adv. Math. 226 (2011), no. 3, 2118-2168.

\bibitem{Kikuta}K.~Kikuta,
\emph{Curvature of the space of stability conditions},
arXiv:1907.10973.

\bibitem{KQ15}A.~King and Y.~Qiu,
\emph{Exchange graphs and Ext quivers},
Adv. Math. 285 (2015), 1106-1154.

\bibitem{KS}M.~Kontsevich and Y.~Soibelman,
\emph{Stability structures, motivic Donaldson--Thomas invariants and cluster transformations},
arXiv:0811.2435.

\bibitem{Li}C.~Li,
\emph{On stability conditions for the quintic threefold},
Invent. Math. 218 (2019), no. 1, 301-340.

\bibitem{MP15}A.~Maciocia and D.~Piyaratne,
\emph{Fourier--Mukai transforms and Bridgeland stability conditions on abelian threefolds},
Algebr. Geom., 2 (2015), no. 3, 270-297.

\bibitem{MacriExceptional}E.~Macr\`i,
\emph{Some examples of spaces of stability conditions on derived categories},
arXiv:math/0411613.

\bibitem{MacriCurve}E.~Macr\`i,
\emph{Stability conditions on curves},
Math. Res. Lett., 14 (2007), no. 4, 657-672.

\bibitem{Okada}S.~Okada,
\emph{Stability manifolds on $\P^1$},
J. Algebraic Geom., 15 (2006), no. 3, 487-505.

\bibitem{Ouchi}G.~Ouchi,
\emph{On entropy of spherical twists},
Proc. Amer. Math. Soc. 148 (2020), no. 3, 1003-1014.

\bibitem{QW}Y.~Qiu and J.~Woolf,
\emph{Contractible stability spaces and faithful braid group actions},
Geom. Topol. 22 (2018), no. 6, 3701-3760.

\bibitem{Rivin}I.~Rivin,
\emph{Walks on groups, counting reducible matrices, polynomials, and surface and free group automorphisms},
Duke Math. J. 142 (2008), no. 2, 353-379.

\bibitem{ST}P.~Seidel and R.~Thomas,
\emph{Braid group actions on derived categories of coherent sheaves},
Duke Math. J. 108 (2001), no. 1, 37-108.

\bibitem{Th2}W.~P.~Thurston,
\emph{Three-dimensional manifolds, Kleinian groups and hyperbolic geometry},
Bull. Amer. Math. Soc. (N.S.) 6 (1982), no. 3, 357-381.

\bibitem{Th}W.~P.~Thurston,
\emph{On the geometry and dynamics of diffeomorphisms of surfaces},
Bull. Amer. Math. Soc. (N.S.) 19 (1988), no. 2, 417-431. 




\end{thebibliography}
\end{document}